\documentclass{amsart}
\usepackage{amssymb,latexsym}
\usepackage[mathscr]{eucal}

\theoremstyle{plain}
\newtheorem{theorem}{Theorem}[section]
\newtheorem{corollary}[theorem]{Corollary}
\newtheorem{lemma}[theorem]{Lemma}
\newtheorem{proposition}[theorem]{Proposition}
\newtheorem{remark}[theorem]{Remark}
\newtheorem{definition}[theorem]{Definition}
\newtheorem{example}[theorem]{Example}

\def \j*{j\in \mathbb J}
\def \o*{operator-valued frame}
\def \f*{$\{A_j\}_{j\in \mathbb J}$}

\newcommand{\be}{\begin{equation}\label}
\newcommand{\ee}{\end{equation}}
\newcommand{\bq}{\begin{equation*}}
\newcommand{\eq}{\end{equation*}}
\newcommand{\bp}{\begin{proof}}
\newcommand{\ep}{\end{proof}}
\newcommand{\bL}{\begin{lemma}\label}
\newcommand{\eL}{\end{lemma}}
\newcommand{\bP}{\begin{proposition}\label}
\newcommand{\eP}{\end{proposition}}
\newcommand{\bC}{\begin{corollary}\label}
\newcommand{\eC}{\end{corollary}}
\newcommand{\bT}{\begin{theorem}\label}
\newcommand{\eT}{\end{theorem}}
\newcommand{\bR}{\begin{remark}\label}
\newcommand{\eR}{\end{remark}}
\newcommand{\bD}{\begin{definition}\label}
\newcommand{\eD}{\end{definition}}
\begin{document}

\title{Operator Valued Frames}
\author{Victor Kaftal}
\address{University of Cincinnati\\ 
         Department of Mathematics\\
         Cincinnati, OH, 45221, 
         USA} 
\email{victor.kaftal@uc.edu}

\author{David H. Larson}
\address{Texas A\&M University\\
									Department of Mathematics\\
									College Station, TX 77843, 
								 USA}
\email{larson@math.tamu.edu}

\author{Shuang Zhang}\thanks{The research of the first and third named
author were supported in part by grants of the Charles Phelps Taft Research Center. The research of the second author was supported in part by grants of 
the National Science Foundation. The first and the second named authors participated in NSF supported Workshops in Linear Analysis and Probability, Texas A\&M University}
\address{University of Cincinnati\\ 
         Department of Mathematics\\
         Cincinnati, OH, 45221, 
         USA} 
\email{zhangs@math.uc.edu}

\date{July 21, 2007}

\begin{abstract}
We develop a natural generalization of vector-valued frame
theory, we term operator-valued frame theory, using operator-algebraic
methods. This extends work of the second author and D. Han which can be viewed 
as the multiplicity one case and extends to higher multiplicity their dilation approach.
We prove several results for operator-valued frames concerning duality, disjointeness, complementarity , and  composition of operator valued frames and the relationship between the two types of  similarity (left and right) of such frames.  A key technical tool is the parametrization of Parseval operator valued frames in terms of a class of partial isometries in the Hilbert space of the analysis operator.
   We apply these notions to an analysis of multiframe generators for the
action of a discrete group $G$ on a Hilbert space.  
One of the main results of the Han-Larson work was the parametrization of
the Parseval frame generators in terms of the unitary operators in the
von Neumann algebra generated by the group representation, and the resulting norm
path-connectedness of the set of frame generators due to the connectedness
of the group of unitary operators of an arbitrary von Neumann algebra.
 In this paper we generalize this multiplicity one result to operator-valued frames.
However, both the parameterization and the proof of norm
path-connectedness turn out to be necessarily more complicated, and this
is at least in part the rationale for this paper.  Our parameterization
involves a class of partial isometries of a different von Neumann algebra.
These partial isometries are not path-connected in the norm topology, but
only in the strong operator topology.  We prove that the set of operator
frame generators is norm pathwise-connected precisely when the von Neumann
algebra generated by the right  representation of the group has no minimal
projections.  As in the multiplicity one theory there are analogous
results for general (non-Parseval) frames.  
\end{abstract}
\maketitle

\section{Introduction} \label{S:intro}

The mathematical theory of frame sequences on Hilbert space has developed rather rapidly in the past decade. This is true of both the finite dimensional and infinite dimensional aspect of the theory. The motivation has come from applications to engineering as well as from the pure mathematics of the theory. 

The theory of finite frames has developed almost as a separate theory in itself, with applications to industry (cf. the recent work \cite {CBE06} of Balans, Casazza, and Edidin on signal reconstruction without noisy phase) as well as recently demonstrated connections to theoretical problems such as the Kadison-Singer Problem \cite {CFTW06}. 

Important examples of infinite frames are the Gabor (Weyl-Heisenberg)  frames of time-frequency analysis and the wavelet frames \cite {cO03}.  Some papers dealing with infinite frames which relate directly or indirectly to this article are \cite {HL00, GH03, DL98, AILP, DFKLOW, KL04, KLZ, dL04}.

Work on this article began in January 1999, when the first-named author visited the second-named author at Texas A\&M University following the special session on ``The functional and harmonic analysis of wavelets and frames" that took place at the annual AMS meeting at San Antonio. Our purpose was to develop operator theoretic methods for dealing with multiwavelets and multiframes, thus extending the approach of the AMS Memoir \cite {HL00}. We developed the theory of operator-valued frames to provide a framework for such problems and we will test this model by solving a problem concerning norm path-connectedness. It has been brought to our attention that a few other recent papers in the literature overlap to some extent with our approach,  notably works of Casazza, Kutyniok and Li \cite {CKL}  on "fusion frames", and also recent work of Bodmann \cite {bB07} on quantum 
computing and work of W. Sun \cite {wS06} on g-ftrames.  These do not deal however with the path connectedness that we address. 
The papers of Kadison on the Pythagorean theorem \cite {rK102, rK202}
are examples of works of pure mathematics that several authors have
realized are both directly and indirectly relevant to frame theory.
They contain theorems on the possible diagonals of positive
operators both in B(H) and in  von Neumann algebras.  This topic is
closely related to the topic of rank-one decompositions and more
general summation decompositions of positive operators, and resolutions of
the identity operator, as investigated in \cite {DFKLOW, KL04} for its relevance to
frame theory.

Also, several papers in the literature deal with frames in Hilbert $C^*$-modules, including one by the same authors of this paper \cite {FL99, FL03, KLZ}. The problems and framework considered in this paper are of a significantly different nature and there is no 
essential overlap.

We note that the key idea in \cite {HL00}, was the observation that frames "dilate" 
to Riesz bases. This was proven at the beginning of [ \cite {HL00} (see also \cite[p. 145] {dL06} ), and was then used to obtain results on Gabor frames, more generally frames 
generated by the action of unitary systems, and certain group 
representations.  The dilation result for the special case of Parseval 
frames can be simply deduced from Naimark's dilation theorem for positive 
operator valued measures, in fact from the special case of Naimark's 
Theorem specific to purely atomic positive operator valued measures.  W. 
Czaja gives a nice account of this in \cite {wC03}, along with some new
dilation results.  V. Paulsen gives a nice proof of Naimark's theorem in \cite {vP02}
using the theory of completely positive mappings.
Similarly, we use dilation theory in the present paper to work with
operator-valued frames.

Consider a \emph{multiframe generator} $\{\psi_1,\psi_2\}$ for a  unitary system  $\mathscr U$, that is two vectors  in a Hilbert space $H$ for which the collection $\{U\psi_m \mid U\in \mathscr U, m=1,2\}$ forms  a frame: 
\be {e:multiframe}
a\|x\|^2 \le \sum_{U\in \mathscr U}\biggl( |(x,U\psi_1)|^2 +|(x,U\psi_2)|^2 \biggr) \le b\|x\|^2.
\ee
for some positive constants $a$ and $b$ and all $x\in H.$
Set $H_o:=\mathbb C^2$, choose $\{e_1, e_2\}$ to be an orthonormal basis of $H_o$, define the rank-two operator $A$ given for  $z\in H$ by $Az :=   (z,  \psi_1)e_1 + (z,  \psi_2)e_2$   and then denote $A_U:=AU^*$. Then equation (\ref{e:multiframe}) holds precisely when
\[
aI\le \sum_{U\in \mathscr U}A_U^*A_U\le bI,
\]
where the series converges in the strong operator topology. In other words, in lieu of considering the two
vectors $\{\psi_1,\psi_2\}$, we can consider the rank-two operator $A$. 

The above is a simple example of an \emph{operator valued frame} generator and leads naturally to the more general Definition \ref {D:oper frame} below of an \o* consisting of operators with ranges in a given Hilbert space $H_o$ and the frame condition is expressed in terms of  boundedness above and below of a series  of positive operators converging in the strong operator topology.  So, the usual (vector) frames can be seen as \o*s of ``multiplicity one".

It is easy to recover from the operator $A$ defined above the vectors $\psi_1$ and $\psi_2$ that were
 used to define it, and, in general, to decompose (but not uniquely) an \o* in a (vector) multiframe 
(see comments after Remark \ref {R:compos}.)  However, we expect that this paper will make clear that ``assembling" 
a multiframe in an \o* is more than just a space-saving bookeeping device. 

Indeed, Operator Theory techniques permit to obtain directly for an \o* 
(and hence for the related  vector multiframe)  properties known for (vector) frames. 

More importantly, however, treating multiframes as \o*s permits to parametrize them in an explicit and transparent way and thus handle the sometimes major differences that occur when the multiplicity rises above one, and in particular,  when it is infinite. 

A case in point, and in a sense our best ``test" of the usefulness of the notion of 
\o*s, is the analysis of frame generators for a discrete group (see \cite {HL00} and Section \ref{S:groups}
 for  a review of the definitions). Han and the second named author proved in \cite [Theorem 6.17]{HL00} that  the collection of all the Parseval frame generators for a given unitary representation $\{G, \pi, H\}$
 of a countable group $G$ is (uniquely)  parametrized by the unitary operators of the von Neumann algebra $\pi(G)''$ generated by the  unitaries $\pi_g$ of the representation. Since the unitary group of any von Neumann  algebra is path-connected in the norm topology, the collection of all the Parseval frame generators is therefore also path-connected, i.e., it has a single homotopy class.

As soon as $\dim H_o > 1$, the above parametrization is no longer sufficient (see Remark \ref {R: other param}),  and it must be replaced by a new parametrization involving a class of partial isometries of a different  von Neumann algebra (see Theorem 7.1, Proposition \ref{P: right group sim}).

Furthermore, when $\dim H_o =\infty$, it is possible to show that the partial isometries involved in this 
parametrization are not path-connected in the norm topology (they are path-connected in the strong operator  topology, though). Nevertheless, we prove in (Theorem \ref{T:homot}) that the collection of operator frame generators is still norm connected, precisely when  the von Neumann algebra generated by the left (or right) regular representation of the group has no minimal projections. The key step is provided by Lemma \ref {L:homot}  where the  strong continuity of a certain path of partial isometries is parlayed into the norm continuity of the  corresponding path of operator frame generators.

One of the main themes of this article is the analysis of one-to-one parametrizations of \o*s in general, and  of operator frame generators for unitary systems and groups in particular. In the process we extend to \o*s  many of the properties of vector frames. More in detail:

In Section \ref {S:frames} we define \o*s, their analysis operators and their frame projections, and then prove that  the dilation approach of \cite{HL00} carries over to the higher multiplicity case, i.e., that Parseval \o*s are the compressions of ``orthonormal" \o*s, namely collections of partial isometries, all with the same initial projection and with  mutually orthogonal ranges spanning the space. 

In Section \ref{S:param} we obtain a one-to-one parametrization of all the \o*s on a certain Hilbert space $H$, with given multiplicity and index set, in terms of a class of operators in the analysis operator Hilbert space  (partial isometries if we consider only Parseval frames). 

In Section  \ref{S:simil} we study two kind of similarities of \o*s. The similarity obtained by multiplying an \o* from the right generalizes the one usual in the vector case and inherits its  main properties. For higher multiplicity, however, we have also a similarity from the left which has  different properties. We characterize the case when two \o*s are similar both from the right and from the left,  in terms of the parametrization mentioned above (Proposition \ref {P:left right}). We also discuss composition of frames,
 when the range of the operators in one frame matches the domain of the operators in a second frame and  we present this notion as the tool to decompose an \o* into a (vector) multiframe.

In Section \ref{S:dual} we define and parametrize the dual of \o*s and extend to higher multiplicity also the notions of 
disjoint, strongly disjoint, and strongly complementary frames that were introduced in \cite {HL00} for the vector case.

In Section \ref{S:groups} we start the analysis of operator frame generators for unitary systems. The notion of
 \emph{local commutant} introduced in \cite{DL98} has a natural analog in the higher multiplicity case (see \ref{P:comm}.) 
Unitary representations of discrete  groups have an operator frame generator with values in
$H_o$ precisely when they are unitarily equivalent to a subrepresentation of the left regular representation
with multiplicity $\dim H_o$, i.e., $\lambda \otimes I_o$ with $I_o$ the identity of $B(H_o)$ (Theorem \ref {T:repres}). 
This result was previously formulated in terms of (vector) multiframes in \cite [Theorem 3.11]{HL00}.

In Section \ref {S:param gen} we present parametrizations of operator frame generators for a discrete
 group represention (Theorem \ref {T:param gen}). As mentioned above, higher multiplicity brings substantial 
differences with the vector case, which are illustrated by Proposition \ref {P:group simil}.

As already mentioned, Section \ref{S:homot} studies the path-connectedness of the operator frame generators for
a unitary representation of a discrete group using von Neumann algebras techniques. 

Finally, let us notice explicitly that although in the applications, frames are mainly  indexed by finite or countable index sets and the vectors in a frame belong to finite or separable Hilbert spaces, and similarly,  discrete groups are finite or countable, we found that making these assumptions provides no simplification in our proofs (with one very minor exception).  Thus we decided to state and prove our results in the general case. The only thing  to keep in mind when the index set $\mathbb J$ is not finite or $\mathbb N$, is that the convergence of $\sum_{j\in \mathbb J}x_j$ means the convergence of the net of the finite partial sums  for all  finite subsets of $\mathbb J$.

\section{Operator valued frames} \label{S:frames}

\begin{definition}\label{D:oper frame}

Let $H$ and $H_o$ be Hilbert spaces. A collection  $\{A_j\}_{j\in \mathbb J}~$
of operators $A_j\in B(H,H_o)$ indexed by $\mathbb J$ is called an operator-valued frame on $H$
with range in $ H_o$ if the series
\begin{equation} \label{eq:D}
S_A:=\sum_{j\in \mathbb J}~A_j^*A_j
\end{equation}\label{eq:def}
converges in the strong operator topology to a positive bounded invertible operator $S_A$. 
The frame bounds $a$ and $b$ are the largest number $a >0$ and the smallest number $b>0$ for which $aI\le S_A\le bI$. If $a=b$, i.e., $S_A=aI$, then   the frame is called tight; if $S_A=I$,  the frame is called  Parseval.  $\sup \{ rank (A_j )\mid j\in \mathbb J\}$ is called the multiplicity of the \o*. When the reference to $H,~H_o,$ and $\mathbb J$ is understood, we denote by $\mathscr F$ the set of all the operator-valued frames on $H$, with ranges in $H_o$  and indexed by $\mathbb J$.

\end{definition}

If the \o* has multiplicity one,  the operators $A_j$ can be identified through the Riesz Representation Theorem with Hilbert space vectors and hence in this case an \o* is indeed a (vector) frame under the usual definition. Explicitly, 
if $A_j$ is the rank one operator given by $A_j z= (x,x_j)e_j$ for some unit vectors $e_j\in H_o$, some vectors $x_j \in H$ and all $z\in H$, then $S_A=\sum_{j\in \mathbb J}~A_j^*A_j j$, hence $S_A$ is bounded and invertible
 if and only if $aI \le S_A \le bI$ for some $a,\, b >0$, namely,
\[
a\|x\|^2 \le (S_Ax,x) = \sum_{j\in \mathbb J}~|(x,x_j)|^2  \le a\|x\|^2
\]
for all $x\in H$. This is precisely the condition that guarantees that  $\{x_j\}_{j\in \mathbb J}$ is a (vector) frame.

Notice that if $\{A_j\}_{j\in \mathbb J}~\in \mathscr F$ and if $\{e_m\}_{m\in \mathbb M}$ is
an orthonormal basis of $H_o$, then it is easy to see that $\{A_j^*e_m\}_{(j,m)\in
\mathbb{J}\times\mathbb{M}}$ is a (vector) frame on $H$, i.e.,
operator-valued frames can be decomposed into (vector) multiframes.
We will revisit this decomposition when discussing more generally frame compositions.

The advantage of treating a collection of vectors forming a multiframe as an operator-valued frame is that we can more easily apply to it the formalism of operator theory. This already evidenced by the next example.

\begin{example}\label{E:dil} Let $K$ be an infinite dimensional
Hilbert space and let  $\{V_n\}_{n\in \mathbb N}$
be a collection partial isometries with mutually orthogonal  range projections $V_nV_n^*$ summing to the identity and all with the same initial projection $V_n^*V_n=E_o$. Let $P\in B(K)$ be a nonzero projection and  let $A_n:= V_n^*P\in B(H,H_o)$ where we set  $H:=PK$, $H_o :=E_oK$. Then
\[
\sum_{n=1}^{\infty}~A_n^*A_n=P(\sum_{n=1}^{\infty}~V_nV_n^* )\left. P \right|_{PK}  =
\left. P \right|_{PK} = I,
\]
i.e., the sequence  $\{A_n\}$ is a Parseval frame with range in $H_o$.
\end{example}

By introducing the \emph{analysis operator} (also called frame transform, e.g., \cite{HL00}), we will see that this example is `generic' (see Proposition \ref{P:dil} below.)

\subsection*{Analysis operator}
Given a Hilbert space $H_o$ and an index set $\mathbb J$, define the
partial isometries
\begin{equation}\label{eq:def L}
  L_j\medspace:\medspace H_o \owns h \mapsto e_j\otimes h \in \ell(\mathbb J)\otimes H_o
\end{equation}
where $\{e_j\}$ is the standard basis of $\ell^2(\mathbb J)$. Then
\begin{equation}\label{eq:L}
L_j^*L_i=
\begin{cases}
I_o  \quad &\text{if } i=j\\
0 \quad &\text{if } i\ne j
\end{cases}
\end{equation}
and
\begin{equation}\label{eq:orthog L}
\sum_{j\in \mathbb J}~L_jL_j^*=I\otimes I_o
\end{equation}
where $I$ denotes the identity operator on $\ell^2(\mathbb J)$ and
$I_o$ denotes the identity operator on $H_o$.

\begin{proposition}\label{P:transform}
For every $\{A_j\}_{j\in \mathbb J}\in \mathscr F$, 
\item[(i)] The series
$ \sum_{{j\in \mathbb J}}L_jA_j$ converges in the strong operator topology to  an operator
$\theta_A \in B(H,\ell(\mathbb J)\otimes H_o)$
\item[(ii)]
$S_A=\theta_A^*\theta_A $
\item[(iii)] $\{A_j\}_{j\in \mathbb J}$ is Parseval if and only if $\theta_A$ is an isometry
\end{proposition}

\begin{proof}
 For every $x\in H$
\[
\|\theta_Ax\|^2 =\sum_{j\in \mathbb J} \|L_jA_jx\|^2
=\sum_{j\in \mathbb J} \|A_jx\|^2
=(S_Ax,x),
\]
where the first identity holds because the operators $L_j$ have
mutually orthogonal ranges, the second one holds because
they are isometries, and the third one holds by the definition (\ref{eq:D}) of $S_A$.
These identities and routine arguments prove (i)-(iii).

\end{proof}

Explicitly, 
\be{eq:transf}
\theta_A =  \sum_{{j\in \mathbb J}}L_jA_j
\ee
and for every $x\in H$, $\theta_A(x)= \sum_{\j*}\,(e_j\otimes A_jx)$. 
As a consequence of Proposition \ref {P:transform} (ii),
\begin{equation}\label{eq:isom} \theta_A S_A^{-1/2} \quad\text{is an isometry}
\end{equation}
 and hence
\begin{equation}\label{eq:P}
P_A: =\theta_A S_A^{-1}\theta_A^*
\end{equation}
is the range projection of $\theta_AS_A^{-1/2}$ and hence of $\theta_A$. Moreover,
 $\{A_j\}_{j\in \mathbb J}$ is Parseval if and only if $\theta_A\theta_A^*$ is a projection.
 
Given $\{A_j\}_{j\in \mathbb J}\in\mathscr F$, the operator 
$\theta_A\in  B(H, \ell(\mathbb J)\otimes H_o)$ is called the \emph{analysis operator}
 and the projection $P_A\in B( \ell(\mathbb J)\otimes H_o)$ is called the
\emph{frame projection} of $\{A_j\}_{j\in \mathbb J}$.

The analysis operator fully `encodes' the information carried by the
operator-valued frame, namely the frame can be  \emph{reconstructed} from its analysis operator via the identity
\begin{equation}\label{eq:reconstr}
A_j=L_j^*\theta_A \quad \text{for all $j\in \mathbb J$}.
\end{equation}
Indeed, 
\[
L_j^*\theta_A = L_j^*\sum_{i\in \mathbb J}~L_iA_i  =
\sum_{i\in \mathbb J}~L_j^*L_iA_i = A_j
\]
by (\ref{eq:L}). In particular, two operator-valued frames $\{A_j\}_{j\in
\mathbb J}$ and $\{B_j\}_{j\in \mathbb J}\in\mathscr F$ are identical if and only if $\theta_A=\theta_B$.
Also, 
\be{e:transf*}
\theta_A^*=\sum_{i\in \mathbb J}~A_j^*L_j^*
\ee
where the convergence is in the strong topology, because by (\ref{eq:reconstr}), $A_j^*= \theta_A^*L_j$ and $L_jL_j^*$ are 
mutually orthogonal projections that sum to the identity $I\otimes I_o$ of $\ell(\mathbb J)\otimes H_o$. The same argument shows that
for any two \o*s $\{A_j\}_{j\in \mathbb J}$ and $\{B_j\}_{j\in \mathbb J}\,$,
\be{e:2transf}
\theta_B^*\theta_A = \sum_{i\in \mathbb J}~B_j^*A_j \quad \text {in the strong operator topology}
\ee

Recall now that Parseval (vector) frames were shown in \cite[Proposition
1.1]{HL00} to be compressions of orthonormal
bases. The higher multiplicity analog of that result is given by the following proposition.

\begin{proposition}\label{P:dil}
For every $\{A_j\}_{j\in \mathbb J}\in \mathscr F$,
there is a Hilbert space $K$ containing $H$ and $H_o$, a collection of
partial isometries $\{V_j\}_{j\in \mathbb J}$
all with the same domain $H_o$ and with mutually orthogonal ranges
spanning $K$,
and a positive invertible operator $T\in B(H)$ such that $A_j=V_j^*T$
for all $j\in \mathbb J$.
In particular, if $\{A_j\}_{j\in \mathbb J}$ is a Parseval frame,
then $T$ can be chosen to be the projection on $H$.
\end{proposition}
\begin{proof}
Let $K=\ell(\mathbb J)\otimes H_o$. Identify $H_o$ with $\mathbb C
\otimes H_o \subset K$ and identify
$H$ with its image $P_AK$ under the isomorphism $\theta_A S_A^{-1/2}$
(see (\ref{eq:isom})). Then for all $j\in \mathbb J$,
we identify $A_j$ with
\[\left. A_j(  \theta_A S_A^{-1/2})^*\right |_{P_A\negmedspace K} = \left.L_J^*(
\theta_AS_A^{-1/2}\theta_A^*)\right |_{P_A\negmedspace K}.
\]
Clearly, $T:=\theta_AS_A^{-1/2}\left.\theta_A^*\right|_{P_A\negmedspace K}$ is positive and invertible and it is the
identity $\left.P_A\right|_{P_A\negthickspace K}$ on $H$ if and only if
$\{A_j\}_{j\in \mathbb J}$ is a Parseval frame.
\end{proof}

In analogy with the vector case (see \cite [Chaper 1]{HL00}), we introduce the following terminology.

\bD{D:Riesz}
An \o* $\{A_j\}_{j\in \mathbb J}$ for which $P_A = I\otimes I_o$ is called a Riesz frame and an \o* that is 
both Parseval and Riesz, i.e., such that $\theta_A^*\theta_A = I$ and 
$\theta_A\theta_A^*= I\otimes I_o$, is called an orthonormal frame. 

\eD

\bR{R:Riesz}
If we identify \o*s with their images in the analysis space $\ell(\mathbb J)\otimes H_o$, (i.e., up to right unitary equivalence 
of the frames, in the notations of Section \ref{S:simil} below), then
\item[(i)] General frames are the frames of the form $\{L_j^*T \}_{j\in \mathbb J}$ for some positive operator $T=PTP$
invertible in $B(P\ell(\mathbb J)\otimes H_o)$ and some projection $P \in B(\ell(\mathbb J)\otimes H_o)$. The operator $T$ and the projection
$P$ are uniquely determined (up to Hilbert space isomorphism, i.e., right unitary equivalence of the \o*s). 
\item[(ii)] Parseval frames are the frames of the form 
$\{L_j^*P\}_{j\in \mathbb J}$ for some projection $P$ in $B(\ell(\mathbb J)\otimes H_o)$. 
\item[(iii)] Riesz frames are the frames of the form 
$\{L_j^*T\}_{j\in \mathbb J}$ for some invertible operator $T\in B(\ell(\mathbb J)\otimes H_o)$. 
\item[(iv)] $\{L_j^*\}_{j\in \mathbb J}$ is the unique orthonormal frame.

In particular, in the notations of Section \ref{S:simil}, Riesz \o*s are the frames that are right-similar to an
 orthonormal frame (cfr.\cite [Proposition 1.5] {HL00} ).
 \eR
 
 \bR{R:multiplicity}
Notice that in the definition of an \o* $\{A_j\}_{j\in \mathbb J}$  there is no request for $ H_o$ to be ``minimal", i.e., for $\dim H_o$ to coincide with the multiplicity of the frame, i.e., with $\sup \{ rank~A_j \mid j\in \mathbb J\}$. In fact we can consider the operators $A_j$ as having range in a ``larger" Hilbert space, e.g., in $H$ itself. Doing so will produce a new analysis operator into a ``larger"  space, however, both analysis operators will carry the same information about the original frame, which can be recovered equally well from either of them.
 
\eR
\section{Parametrization of operator-valued frames}\label{S:param}
In this section we show that all the operator-valued frames with the same multiplicity and same index set, operating
 on the same Hilbert space (up to isomorphism) can be ``obtained'' from a single operator-valued frame. Consider
 first Parseval frames. Following the dilation viewpoint  (cfr. Proposition \ref {P:dil} above), we can immerse all
 these frames in the analysis space $\ell(\mathbb J)\otimes H_o$ by identifying them with the compression of 
$\{L_j^*\}_{j\in \mathbb J}$ to their frame projection. Since all the frame projections are equivalent
 (each range is isomorphic to the original  Hilbert space of the frame), the partial isometries implementing  these equivalences will provide the link between the frames. To make this idea precise, and to handle at the same  time frames that are not Parseval, we introduce the following notation.

Given $\{A_j\}_{j\in \mathbb J} ~\in \mathscr F$, define
\begin{equation}\label{eq:M}
\mathscr {M}_A:=\{M\in B(\ell(\mathbb J)\otimes H_o) \mid M=MP_A,
 \left.M^*M\right |_{P_A\negmedspace \ell(\mathbb J)\otimes H_o}\text{ is invertible}\}.
\end{equation}
Equivalently,  $\mathscr {M}_A:=\{M\in B(P_A\ell(\mathbb J)\otimes H_o, \ell(\mathbb J)\otimes H_o) \mid M\text{ is left invertible}\}.$
If $M\in \mathscr {M}_A$, denote by $(M^*M)^{-1}\in B(P_A\ell(\mathbb J)\otimes H_o)$ the inverse of $\left.M^*M\right
|_{P_A\negmedspace \ell(\mathbb J)\otimes H_o}$.

\begin{theorem}\label{T:param} Let $\{A_j\}_{j\in \mathbb J} ~\in \mathscr F$. For all 
$\{B_j\}_{j\in \mathbb J} ~\in \mathscr F$ define 
\begin{equation}\label{eq:Phi}
\Phi_A(\{B_j\}_{j\in \mathbb J}~) :=\theta_BS_A^{-1}\theta_A^*.
\end{equation}
Then $\Phi_A: \mathscr F \mapsto \mathscr {M}_A$ is one-to-one and onto and
 $ \Phi_A^{-1}(M) = \{L_j^*M\theta_A\}_{j\in \mathbb J~}$ for all $M\in \mathscr M_A $. 
If  $\{B_j\}_{j\in \mathbb J} ~\in \mathscr F$ and $M:= \Phi_A(\{B_j\}_{j\in \mathbb J})$,  then 
$\theta_B= M\theta_A$ and  
\[
V_M:=M(M^*M)^{-1/2}
\]
is a partial isometry that implements the equivalence $P_B\sim P_A$, i.e., $P_B= V_MV_M^*$ and $P_A= V_M^*V_M$.
\end{theorem}

\begin{proof}
Let  $\{B_j\}_{j\in \mathbb J}\in  \mathscr F $  and let 
$M:=\Phi_A(\{B_j\}_{j\in \mathbb J})$. Then $MP_A=M$
because $P_A$ is the range projection of $\theta_A$. Moreover,
\[
M^*M = \theta_AS_A^{-1}\theta_B^*\theta_BS_A^{-1}\theta_A^*
\ge \theta_AS_A^{-1}aIS_A^{-1}\theta_A^* \ge
\frac{a}{b}\theta_AS_A^{-1}\theta_A^* = \frac{a}{b}P_A,
\]
where $a$ is a lower bound for $S_B$ and $b$ is an upper bound for
$S_A$. Thus $M^*M$ is invertible in $B(P_A\ell(\mathbb J)\otimes H_o)$ and hence $M\in \mathscr {M}_A$, i.e,
 $\Phi_A$ maps into $\mathscr {M}_A$.
Assume now that  $\Phi_A(\{B_j\}_{j\in \mathbb J}) = \Phi_A(\{C_j\}_{j\in \mathbb J})$ for two frames in $\mathscr F$. 
Since $\theta_B=\theta_BS_A^{-1}\theta_A^*\theta_A = M\theta_A  $, it follows that
$\theta_B=\theta_C$, and by (\ref{eq:reconstr}), the two frames coincide, i.e., $\Phi_A$ is injective.
We prove now that $\Phi_A$ is onto. For any $M\in \mathscr {M}_A$, define $\{B_j:=L_j^*M\theta_A\}_{j\in \mathbb J}$.
Then 
\[
\sum_{j\in \mathbb J}B_j^*B_j = \sum_{j\in \mathbb J}\theta_A^*M^*L_jL_j^*M\theta_A = \theta_A^*M^*M\theta_A 
\le \|M\|^2\theta_A^*\theta_A = \|M\|^2S_A.
\]
Similarly,
\[
\sum_{j\in \mathbb J}B_j^*B_j  \ge \|(M^*M)^{-1}\|^{-1}S_A,
\]
which proves that $\{B_j\}_{j\in \mathbb J}$ is an operator-valued frame. Moreover, 
\[
\theta_B= \sum_{j\in \mathbb J}L_jB_j= \sum_{j\in \mathbb J}L_jL_j^*M\theta_A = M\theta_A.
\] 
We have just seen that $ \theta_B = \Phi_A(\{B_j\}_{j\in \mathbb J})\theta_A$. Thus
$M\theta_A = \Phi_A(\{B_j\}_{j\in \mathbb J})\theta_A$ and hence 
\[
M= MP_A= M\theta_AS_A^{-1}\theta_A^* = \Phi_A(\{B_j\}_{j\in \mathbb J})\theta_AS_A^{-1}\theta_A^*
=\Phi(\{B_j\}_{j\in \mathbb J}).
\]
This proves that the map $ \Phi_A$ is onto and that 
$ \Phi_A^{-1}(M) = \{L_j^*M\theta_A\}_{j\in \mathbb J~}$ for all $M$ in $\mathscr M_A $.
It remains only to compute $P_B$, which by definition, is the range projection of $\theta_B = M\theta_A$. Since  
$M= MP_A$ and $\theta_A$ is one-to-one (recall that $\theta_AS_A^{-1/2}$ is an isometry), $P_B$ is the range projection
 of $M$.  Now 
\[
V_M^*V_M= (M^*M)^{-1/2}M^*M(M^*M)^{-1/2} =P_A
\]
 (recall that $(M^*M)^{-1}$ denotes the inverse of $M^*M$ in
$B(P_A\ell(\mathbb J)\otimes H_o)$), hence $V_M$ is a partial isometry. Since $(M^*M)^{-1/2}$ is invertible, the range of 
$V_M$ coincides with the range of $M$, and hence $P_B= V_MV_M^*$.

\end{proof}

An easy consequence of Theorem \ref {T:param} and its proof is the following:

\bC {C:product} If $\{A_j\}_{j\in \mathbb J}, \{B_j\}_{j\in \mathbb J}, \text{ and }\{C_j\}_{j\in \mathbb J}$ are 
\o*s, then $\Phi_A(\{C_j\}_{j\in \mathbb J})= \Phi_B(\{C_j\}_{j\in \mathbb J})\Phi_A(\{B_j\}_{j\in \mathbb J})$.
\eC

This corollary shows that  $\Phi_A(\{B_j\}_{j\in \mathbb J})$ behave like a partial isometry with initial projection $P_A$
 and range projection $P_A$. In fact, if both frames are Parseval, $\Phi_A(\{B_j\}_{j\in \mathbb J})$ is precisely a partial isometry with these initial and range projections.

Since every operator-valued frame  is right-similar to a Parseval frame, (see Definition \ref{D:simil} below), we can focus on Parseval frames. For ease of  reference  we present in the following corollary the main result of Theorem \ref{T:param} formulated directly for Parseval frames. 

\begin{corollary}\label{C:param Pars}
Given a  Parseval operator-valued frame $\{A_j\}_{j\in \mathbb J}\in\mathscr F$, then
\[
\{\{L_j^*V\theta_A\}_{j\in \mathbb J} \mid  V\in B(\ell(\mathbb J)\otimes H_o), V^*V=P_A\}.
\]
  is the collection of all Parseval operator-valued frames in $\mathscr F.$ The correspondence is one-to-one: 
if $V\in B(\ell(\mathbb J)\otimes H_o), V^*V=P_A$, and $\{B_j:= L_j^*V\theta_A\}_{j\in \mathbb J}\in\mathscr F$, then 
$V = \theta_B\theta_A^*$.
\end{corollary}
\begin{proof}
We  need only to show that when $\{A_j\}_{j\in \mathbb J}\in\mathscr F$ is Parseval and $M\in \mathscr M_A$,
then the operator-valued frame $\Phi_A^{-1}(M):=\{B_j:L_j^*M\theta_A\}_{j\in \mathbb J}$) is Parseval if
 and  only if $M$ is a partial isometry. This is clear since 
$\theta_B\theta_B^* = M \theta_A\theta_A^*M^*=MP_AM^*=MM^*$ and as remarked after equation (\ref{eq:P}), 
$\{B_j\}_{j\in \mathbb J}$ is Parseval if and only if $\theta_B\theta_B^*$ is a projection. Finally, 
$V = M= \Phi_A(\{B_j\}_{j\in \mathbb J}~) =\theta_B\theta_A^*$ since $S_A^{-1}=I$. 
\end{proof}

\section{Similarity and composition of frames}\label{S:simil}
For operator valued frames with ranges in $H_o$ where $\dim H_o > 1$
there are two natural distinct notions of similarity which are called
`right' and `left'.

\begin{definition}\label{D:simil} Let  $\{A_j\}_{j\in \mathbb
J}$,$\{B_j\}_{j\in \mathbb J}\in \mathscr F$. We say that:
\item[(i)]  $\{B_j\}_{j\in \mathbb J}$ is right-similar (resp., right
unitarily equivalent)
to $\{A_j\}_{j\in \mathbb J}$
if there is an invertible operator (resp., unitary operator) $T\in
B(H)$ such that $B_j=A_jT$ for all $j\in \mathbb J$.
\item[(ii)]  $\{B_j\}_{j\in \mathbb J}$ is left-similar (resp., left unitarily
equivalent)
  to $\{A_j\}_{j\in \mathbb J}$
  if there is an invertible operator (resp., unitary operator) $R\in
B(H_o)$ such that $B_j=RA_j$ for all $j\in \mathbb J$.
\end{definition}

When $\dim H_o=1$, the left similarity is trivial (a multiplication
by a nonzero scalar) and the right similarity is just the usual  similarity of the vector frames (corresponding to the
operators, i.e., functionals, by the Riesz Representation Theorem).

We leave to the reader the following simple results about right similarity.

\begin{lemma}\label{L:rightsim} Let $\{A_j\}_{j\in \mathbb J}\in \mathscr F$ have frame bounds $a$ and $b$,
  i.e., $aI\le S_A \le bI$, let $T\in B(H)$ be an invertible operator,
and let $\{B_j:=A_jT\}_{j\in \mathbb J}$. Then
\item[(i)] $\{B_j\}_{j\in \mathbb J}\in \mathscr F$ and
$\frac{a}{\|T^{-1}\|^2}I\le S_B \le b\|T\|^2I.$ In particular,
  if $T$ is unitary, then $\{B_j\}_{j\in \mathbb J}$ has the same
bounds as $\{A_j\}_{j\in \mathbb J}$. Assuming that $\{A_j\}_{j\in \mathbb J}$ is Parseval,
then  $\{B_j\}_{j\in \mathbb J}$ is Parseval if and only if $T$ is unitary.
\item[(ii)] $\theta_B = \theta_AT$ and $S_B=T^*S_AT$
\item[(iii)]$\Phi_A(\{B_j\}_{j\in \mathbb J})= \theta_ATS_A^{-1}\theta_A^*.$
\end{lemma}

Now we characterize right-similar frames.

\begin{proposition}\label{P:simil} Let $\{A_j\}_{j\in \mathbb J},\{B_j\}_{j\in \mathbb J}\in \mathscr F$ 
and let $M:=\Phi_A(\{B_j\}_{j\in \mathbb J})$. Then the following conditions are equivalent:
\item[(i)]  $B_j=A_jT$ for all $j\in \mathbb J$ for some invertible operator $T\in B(H)$, i.e., $(\{A_j\}_{j\in \mathbb J}$
 and $(\{B_j\}_{j\in \mathbb J}$ are right-similar.
\item[(ii)] $\theta_B=\theta_AT$ for some invertible $T\in B(H)$.
\item[(iii)] $M=P_AMP_A$ is invertible in
$B(P_A\ell(\mathbb J)\otimes H_o)$.
\item[(iv)] $P_B=P_A.$

If the above conditions are satisfied, the invertible operator $T$ in (i) and (ii) is uniquely 
determined and $T= S_A^{-1}\theta_A^*\theta_B.$

In the case that $\{A_j\}_{j\in \mathbb J}$ is Parseval then $\{B_j\}_{j\in \mathbb J}$
 is Parseval if and only if the operator $T$ in (i), or equivalently in (ii), is unitary.  
\end{proposition}

\begin{proof}
\item[(i)] $\Longleftrightarrow$ (ii) One implication is given by
Lemma \ref{L:rightsim} and the other is immediate.

\item[(ii)] $\Longrightarrow$ (iii) We have
$\theta_B=\theta_AT=(\theta_ATS_A^{-1}\theta_A^*)\theta_A$. Let $N:=\theta_ATS_A^{-1}\theta_A^*$, then
\begin{align*}
N^*N&=(\theta_ATS_A^{-1}\theta_A^*)^*(\theta_ATS_A^{-1}\theta_A^*)
= \theta_AS_A^{-1}T^*S_ATS_A^{-1}\theta_A^*\\
&\ge a\|T\|^2\theta_AS_A^{-2}\theta_A^* 
\ge \frac{a\|T\|^2}{b}\theta_AS_A^{-1}\theta_A^* 
= \frac{a\|T\|^2}{b}P_A,
\end{align*}
and $NP_A=(\theta_ATS_A^{-1}\theta_A^*)P_A=N$.
Thus $N\in \mathscr M_A$ and
by the injectivity of the map $\Phi_A$ in Theorem \ref{T:param},
$M=\theta_ATS_A^{-1}\theta_A^*$. We now see that also
$P_AM= M$, i.e., $M=P_AMP_A$. Furthermore,
\[
M(\theta_AT^{-1}S_A^{-1}\theta_A^*)
= \theta_ATT^{-1}S_A^{-1}\theta_A^*
=\theta_AS_A^{-1}\theta_A^*
=P_A,
\]
and similarly, $(\theta_AT^{-1}S_A^{-1}\theta_A^*)M=P_A.$
Thus $M$ is invertible in $B(P_A\ell(\mathbb J)\otimes H_o)$.

\item[(iii)] $\Longrightarrow$ (iv)  By Theorem \ref{T:param}, $P_B=[M]$, 
the range projection of $M$. By hypothesis, $M$ is invertible in $B(P_A\ell(\mathbb J)\otimes H_o)$, hence
$[M] = P_A$.

\item[(iv)] $\Longrightarrow$ (ii) Since
$\theta_B=P_B\theta_B=P_A\theta_B=\theta_A(S_A^{-1}\theta_A^*\theta_B)$,
 it suffices to show that $S_A^{-1}\theta_A^*\theta_B$ has an inverse, namely
$S_B^{-1}\theta_B^*\theta_A$. Indeed,
 \[
S_A^{-1}\theta_A^*\theta_BS_B^{-1}\theta_B^*\theta_A =
S_A^{-1}\theta_A^*P_B\theta_A = S_A^{-1}\theta_A^*P_A\theta_A = S_A^{-1}\theta_A^*\theta_A=I
 \]
 Similarly,
  \[
S_B^{-1}\theta_B^*\theta_AS_A^{-1}\theta_A^*\theta_B =
S_B^{-1}\theta_B^*P_A\theta_B = S_B^{-1}\theta_B^*P_B\theta_B=S_B^{-1}\theta_B^*\theta_B=I.
 \]
The uniqueness is then easily established.

\end{proof}

\begin{remark} \label{R:uniq}
\item [(i)]  For every \o* $\{A_j\}_{j\in \mathbb J}\,$, $\{B_j:=A_jS_A^{-1/2}\}_{j\in \mathbb J}~$ is 
a Parseval \o*. Thus every \o* is right-similar to a Parseval operator-valued frame. Every \o*
right unitarily equivalent to $\{B_j\}_{j\in \mathbb J}$ is also Parseval and right-similar to
$\{A_j\}_{j\in \mathbb J}$.
\item [(ii)] The equivalence of (i) and (iv) is the higher multiplicity analog of 
\cite [Proposition 2.6 and Corollary 2.8]{HL00}.
\end{remark}

Now we consider left similarity.

\begin{lemma}\label{L:left simil}
Let $\{A_j\}_{j\in \mathbb J}$ be an \o*  having frame bounds $a$ and $b$ and let  
$R\in B(H_o)$ be an invertible operator. Then
\item[(i)]
$\{B_j:=RA_j\}_{j\in \mathbb J}$ is  an operator-valued frame and
$\frac{a}{\|R^{-1}\|^2}I\le S_B \le b\|R\|^2I.$ In particular,
  if $R$ is unitary, then $\{B_j\}_{j\in \mathbb J}$ has the same frame bounds as $\{A_j\}_{j\in \mathbb J}$. 
\item[(ii)] $\theta_B = (I\otimes R)\theta_A$,
$S_B=\theta_A^*(I\otimes R^*R) \theta_A$, $P_B= [(I\otimes R)\theta_A]$ is the range  projection of 
$(I\otimes R)\theta_A$, and $~\Phi_A(\{B_j\}_{j\in \mathbb J})= (I\otimes R)P_A$.
\item[(iii)] Assume that  $\{A_j\}_{j\in \mathbb J}$ is Parseval. Then $\{B_j\}_{j\in \mathbb J}$
 is Parseval if and only if $P_A(I\otimes R^*R)P_A= P_A$, if and only if $P_B = (I\otimes R)P_A(I\otimes R^*)$.
In particular, this holds if $R$ is an isometry.
\item[(iv)] $P_B$ is unitarily equivalent to $P_A$.

\end{lemma}
\begin{proof}
\item[(i)] Obvious.
\item[(ii)] For every $R\in B(H_o)$, $h\in H_o$, and $j\in \mathbb J$, from the
definition of $L_j$  we have that
  $L_jRh= e_j\otimes Rh = (I\otimes R)L_jh$, i.e., $L_jR = (I\otimes R)L_j$. Then
\[
\theta_B= \sum_{j\in \mathbb J}L_jRA_j = \sum_{j\in \mathbb J}(I\otimes R)L_jA_j =
(I\otimes R)\theta_A.
\]
  Consequently,
$S_B=\theta_B^*\theta_B = \theta_A^*(I\otimes R^*R)\theta_A.$
Clearly, $M:= (I\otimes R)P_A\in \mathscr M_A$
and since $\theta_B = (I\otimes R)\theta_A=M\theta_A$, by the
injectivity of $\Phi_A$ in Theorem \ref{T:param},
  it follows that $\Phi_A(\{B_j\}_{j\in \mathbb J})=M$. This can also be verified
directly from 
\[
\Phi_A(\{B_j\}_{j\in \mathbb J})=\theta_BS_A^{-1}\theta_A^*
=(I\otimes R)\theta_AS_A^{-1}\theta_A^*= (I\otimes R)P_A.
\]
\item[(iii)] Immediate from (ii).
\item[(iv)] Denote by  $N(X)$ the null projection of the operator $X$. Then
\[
P_B^{\bot} =  [(I\otimes R)(\theta_A]^{\bot} 
=[(I\otimes R)(\theta_A\theta_A^*)^{1/2}]^{\bot}
=N\left ((\theta_A\theta_A^*)^{1/2}(I\otimes R^*)\right ).
\]
Now $x \in N\left ((\theta_A\theta_A^*)^{1/2}(I\otimes R^*)\right )$ if and only if 
$x \in (I\otimes R^*)^{-1}N\left ((\theta_A\theta_A^*)^{1/2}\right )$, thus 
\[
P_B^{\bot}
= (I\otimes R^{-1})^*P_A^{\bot}
 = [(I\otimes R^{-1})^*P_A^{\bot}]
\sim [P_A^{\bot}(I\otimes R^{-1})]= P_A^{\bot},
\]
where we use the well known fact $[X]\sim [X^*]$.
Since  $P_B \sim P_A$ (e.g., see Theorem \ref{T:param}), it follows that $P_B$ and $P_A$ are unitarily equivalent.

\end{proof}

In Proposition \ref {P:simil} we have seen that the invertible operator implementing the right similarity of two
operator-valued frames is uniquely determined and that it must be a unitary operator when both frames are Parseval.
The following example shows that neither of these conclusions hold in the case of left similarities. 

\begin{example}\label{E:nonuniq}. Let $H:=\ell^2(\mathbb J)$,
$H_o:=\mathbb C^2$, $Q_1,Q_2$ be two orthogonal (rank-one) projections in $B(H_o)$, let 
$P:=I\otimes Q_1$, and let $\{A_j:=L_j^*\left. P\right|_{P  \ell(\mathbb J)\otimes H_o}\}_{j\in \mathbb J}$.
By Example \ref{E:dil}, $\{A_j\}_{j\in \mathbb J}$ is a Parseval operator-valued frame. Moreover,  for every $\lambda \ne 0$,  $R:=Q_1+\lambda Q_2$ is an invertible operator and 
 \begin{align*}
RA_j &=RL_j^*\left. P\right|_{P \ell(\mathbb J)\otimes H_o} =  L_j^*(I\otimes R)\left. P\right|_{P \ell(\mathbb J)\otimes H_o}\\
&= L_j^*(I\otimes (Q_1+\lambda Q_2))\left. (I\otimes Q_1)\right|_{P \ell(\mathbb J)\otimes H_o}
= L_j^*\left. P\right|_{P \ell(\mathbb J)\otimes H_o} = A_j.
\end{align*}
\end{example}

\bigskip

If $A_j\}_{j\in \mathbb J}$ and $\{B_j=RA_j\}_{j\in \mathbb J}$ are two left-similar Parseval \o*s, 
 $P_B = (I\otimes R)P_A(I\otimes R^*)$ by Lemma \ref{L:left simil} and furthermore there is a unitary
 operator $U$ for which $P_B=UP_AU^*$. The following example shows that it can be impossible to choose 
$U\in I\otimes B(H_o)$. 

\begin{example}\label{E:unit}. Let $H:=\ell^2(\mathbb J)$, $H_o:=\mathbb C^2$, $P_1, P_2$ 
be two orthogonal projections in $B(H)$,  
$Q_1:=\begin{pmatrix}
1 & 0\\
0 & 0
\end{pmatrix},
Q_2:=\begin{pmatrix}
0 & 0\\
0 & 1
\end{pmatrix}, 
Q_3:=\begin{pmatrix}
1/ 2 & 1/ 2\\
1/ 2 & 1/ 2
\end{pmatrix}$, and
$R:=\begin{pmatrix}
1 & 1/\sqrt 2\\
0 & 1/\sqrt 2
\end{pmatrix}$. 
Define $P:=P_1\otimes Q_1 + P_2 \otimes Q_2$,  
$\{A_j:=L_j^*\left. P\right|_{P \ell(\mathbb J)\otimes H_o}\}_{j\in \mathbb J}$, and $\{B_j:=RA_j\}_{j\in \mathbb J}$.
Then by Example \ref{E:dil}, $\{A_j\}$ is a Parseval operator-valued frame  and $P_A = P$. 
Since
\[
P(I\otimes R^*R)P= P_1\otimes Q_1R^*RQ_1+P_2\otimes Q_2R^*RQ_2 = P_1\otimes Q_+P_2\otimes Q_2 =P,
\]  
by Lemma \ref{L:left simil} (i) and (iii), $\{B_j\}_{j\in \mathbb J}$ is also a 
Parseval operator-valued frame and 
\begin{align*}
P_B&= (1\otimes R)\theta_A \theta_A^* (1\otimes R^*) 
=  (1\otimes R)P (1\otimes R^*)\\
&=P_1\otimes RQ_1R^* + P_2\otimes RQ_2R^*
= P_1\otimes Q_1 + P_2\otimes Q_3.
\end{align*}
This implies that if $P_B =UPU^*$ for some unitary $U$, then $U\not\in I\otimes B(H_o)$. Indeed
if $ U= 1\otimes W$ for some unitary $W\in B(H_o)$, then by the above computation
\[
P_B= P_1\otimes WQ_1W^* + P_2\otimes WQ_2W^* = P_1\otimes Q_1
+ P_2\otimes Q_3
\]
which is impossible since $WQ_1W^*WQ_2W^*=0$ while $Q_1Q_3\ne 0$.

\end{example}

Operator-valued frames can be both right and left-similar. The following proposition
 determines when this occurs.

\begin{proposition}\label{P:left right} Let $\{A_j\}_{j\in \mathbb J}\in\mathscr F$, $R\in B(H_o)$ 
be an invertible operator, and let $\{B_j:=RA_j\}_{j\in \mathbb J}$. Then the following conditions are 
equivalent.
\item[(i)] $\{B_j\}_{j\in \mathbb J}$ and $\{A_j\}_{j\in \mathbb J}$ are right-similar. 
\item[(ii)] $(I\otimes R)P_A = P_A(I\otimes R)P_A$ is invertible in $B(P_A\ell(\mathbb J)\otimes H_o)$.
\item[(iii)] $P_A^{\bot}(I\otimes R)P_A =0$ and $P_A^{\bot}(I\otimes R^{-1})P_A =0$.
\item[(iv)] 
\begin{equation}\label{eq:13}
RA_i=A_iS_A^{-1}\sum_{\j*}A_j^*RA_j \quad \text{for every} ~~i \in \mathbb J.
\end{equation}
and 
\begin{equation}\label{eq:14}
R^{-1}A_i=A_iS_A^{-1}\sum_{\j*}A_j^*R^{-1}A_j \quad \text{for every} ~~i \in \mathbb J.
\end{equation}
\item[(v)] \begin{equation*}
RA_i=A_iS_A^{-1}\sum_{\j*}A_j^*RA_j \quad \text{for every} ~~i \in \mathbb J.
\end{equation*}
and 
\begin{equation}\label{eq:15}
\biggl (\sum_{\j*}A_j^*RA_j \biggr)^{-1} 
= S_A^{-1}(\sum_{\j*}A_j^*R^{-1}A_j )S_A^{-1}
\end{equation}
If $R$ is unitary, then these conditions are also equivalent to 
\item[(vi)]$(I\otimes R)P_A = P_A(I\otimes R)$.
\item[(vii)] $R$ commutes with $A_jS_A^{-1}A_i^*$ for all $i,\j*$.
 \end{proposition}
\begin{proof}
 
\item[(i)]$\Longleftrightarrow$ (ii) Let $M:=\Phi_A(\{B_j\}_{j\in \mathbb J})$. 
Then by Lemma \ref{L:left simil}, $M= (I\otimes R)P_A$. By Proposition \ref {P:simil},
$\{B_j\}_{j\in \mathbb J}$ is right-similar to $\{A_j\}_{j\in \mathbb J}$ if and only if
 $MP_A=P_AMP_A$ is invertible in $B(\ell(\mathbb J)\otimes H_o)$.

\item[(ii)]$\Longleftrightarrow$ (iii) Set
\[ 
I\otimes R:=\begin{pmatrix}
P_A(I\otimes R)P_A & P_A(I\otimes R)P_A^{\bot}\\
0 & P_A^{\bot}(I\otimes R)P_A^{\bot}
\end{pmatrix}
\quad \text{and} \]
\[
(I\otimes R)^{-1}:=\begin{pmatrix}
P_A(I\otimes R^{-1})P_A & P_A(I\otimes R^{-1}))P_A^{\bot}\\
P_A^{\bot}(I\otimes R^{-1}))P_A & P_A^{\bot}(I\otimes R^{-1}))P_A^{\bot}
\end{pmatrix}.
\]
Then it is immediate to verify that $P_A^{\bot}(I\otimes R^{-1}))P_A=0$ if and only if 
\[
(P_A(I\otimes R)P_A)(P_A(I\otimes R^{-1})P_A) = P_A,
\]
 and if and only if $P_A(I\otimes R)P_A$ is invertible. 

\item[(iii)]$\Longleftrightarrow$ (iv) Recall that $(I\otimes R)L_j=L_jR$ for all $\j*$, and that
$P_A = \theta_AS_A^{-1}\theta_A^*$. Therefore 
\[
(I\otimes R)P_A = (I\otimes R) \theta_AS_A^{-1}\theta_A^* 
=(I\otimes R)\sum_{\j*}L_jA_jS_A^{-1}\theta_A^*
=\sum_{\j*}L_jRA_jS_A^{-1}\theta_A^*
\]
and
\begin{align*}
P_A(I\otimes R)P_A
&=  \theta_AS_A^{-1}\theta_A^*(I\otimes R)\theta_AS_A^{-1}\theta_A^*\\
&= (\sum_{\j*}L_jA_j )S_A^{-1}(\sum_{i,j \in \mathbb J}A_i^*L_i^*(I\otimes R)L_jA_j )S_A^{-1}\theta_A^*\\
&= (\sum_{\j*}L_jA_j )S_A^{-1}(\sum_{i,j \in \mathbb J}A_i^*RL_i^*L_jA_j )S_A^{-1}\theta_A^*\\
&= (\sum_{\j*}L_jA_j )S_A^{-1}(\sum_{\j*}A_j^*RA_j )S_A^{-1}\theta_A^*.
\end{align*}
Hence $P_A^{\bot}(I\otimes R^{-1})P_A =0$ if and only if
\begin{equation}\label{eq:16}
\sum_{\j*}L_jRA_jS_A^{-1}\theta_A^* = (\sum_{\j*}L_jA_j )S_A^{-1}(\sum_{\j*}A_j^*RA_j )S_A^{-1}\theta_A^*.
\end{equation}
By multiplying (\ref{eq:16}) on the left by $L_i^*$ and on the right by $\theta_A$,
we obtain  (\ref{eq:13}). Conversely, by multiplying (\ref{eq:13}) on the left by $L_i$ and on the right 
by $S_A^{-1}\theta_A^*$ and summing over $i\in \mathbb J$ we obtain back (\ref{eq:16}). Thus 
$P_A^{\bot}(I\otimes R^{-1})P_A =0$ is equivalent to (\ref{eq:13}).
By the same argument, $P_A^{\bot}(I\otimes R^{-1})P_A =0$ is equivalent to (\ref{eq:14}).
\item[(iv)]$\Longleftrightarrow$ (v)
Assume that (\ref{eq:13})  and (\ref{eq:14}) hold. Then 
\[
A_i^* A_i = A_i^*R^{-1}A_i S_A^{-1}(\sum_{\j*}A_j^*RA_j ) \quad \text{for all $i$}
\]
and hence by summing over $i\in \mathbb J$ and then multiplying on the left and on the right by $S_A^{-1/2}$  
we obtain
\[
S_A^{-1/2}(\sum_{i\in \mathbb J}A_i^*R^{-1}A_i )S_A^{-1}(\sum_{\j*}A_j^*RA_j )S_A^{-1/2} = I.
\]
Similarly, 
\[
S_A^{-1/2}(\sum_{i\in \mathbb J}A_i^*RA_i )S_A^{-1}(\sum_{\j*}A_j^*R^{-1}A_j )S_A^{-1/2} = I.
\]
Thus 
\[
\biggl (S_A^{-1/2}(\sum_{\j*}A_j^*RA_j )S_A^{-1/2} \biggr )^{-1}
= S_A^{-1/2} (\sum_{\j*}A_j^*R^{-1}A_j ) S_A^{-1/2},
\]
and hence (\ref{eq:15}) holds. Conversely, if (\ref{eq:13}) and (\ref{eq:15})  hold,
 then by multiplying (\ref{eq:13}) on the right by the right hand side of (\ref{eq:15})
 we easily obtain (\ref{eq:14}).

Assume now that $R$ is unitary.
 \item[(iii)]$\Longleftrightarrow$ (vi) Obvious
\item[(vi)]$\Longrightarrow$ (vii) As seen above in course of the proof of  (iii) $\Longleftrightarrow$ (iv), 
\[
(I\otimes R)P_A
=\sum_{i,j\in \mathbb J}L_jRA_jS_A^{-1}A_i^*L_i^*
\] 
and
\[
P_A(I\otimes R)
=
=\sum_{i,j\in \mathbb J}L_jA_jS_A^{-1}A_i^*RL_i^*.
\]
Thus
\[
\sum_{i,j\in \mathbb J}L_jRA_jS_A^{-1}A_i^*L_i^*
=\sum_{i,j\in \mathbb J}L_jA_jS_A^{-1}A_i^*RL_i^*.
\]
By multiplying this identity
 on the left by $L_j^*$ and on the right by $L_i$ and recalling (\ref{eq:L}) we obtain (vii). 
\item[(vii)]$\Longrightarrow$ (iv) For every $i \in \mathbb J$, 
\[
A_iS_A^{-1}\sum_{\j*}A_j^*RA_j= RA_iS_A^{-1}\sum_{\j*}A_j^*A_j
=RA_i,
\]
i.e., (\ref{eq:13}) holds. Since $R^{-1}=R^*$ also commutes with all $A_jS_A^{-1}A_i^*$, the same argument
shows that (\ref{eq:14}) too holds.
\end{proof}

\subsection*{Composition of frames} Let  $A=\{A_j\}_{\j*} $
 be an \o* in $B(H,H_o)$ and $B=\{B_m\}_{m\in \mathbb M}$ be an \o* in $B(H_o, H_1)$. 
Then it is easy to check that $\{C_{(m,j)}:= B_mA_j\}_{m\in \mathbb M, j\in \mathbb J}$ 
is an \o* in $B(H, H_1)$. 
The \o* $\{C_{(m,j)}\}_{(m,j)\in \mathbb M\times \mathbb J}$ is called the composition of
 the frames $A=\{A_j\}_{\j*} $ and $B=\{B_m\}_{m\in \mathbb M}$.

\begin{proposition}\label{P: compos}  Let $\{C_{(m,j)} := B_mA_j\}_{(m,j)\in \mathbb M\times \mathbb J}$
 be the composition of the \o*s $\{A_j\}_{\j*} $
 and $B=\{B_m\}_{m\in \mathbb M}$. Then
\item[(i)]
$\theta _C=(I_{\mathbb J}\otimes \theta _B)\theta _A$ and  
 $S_C=\theta _A^* (I_{\mathbb J}\otimes S_B)\theta _A$ 
where $I_{\mathbb J}$ is the identity on $l^2(\mathbb J)$.
In particular, if $\{A_j\}_{\j*} $ and  $\{B_m\}_{m\in \mathbb M}$ are Parseval, 
then $\{C_{(m,j)}\}_{m\in \mathbb M, j\in \mathbb J}$ is Parseval.
\item[(ii)] If $P_A$ commutes with $(I_{\mathbb J}\otimes S_B)$, then
$P_C= (I_{\mathbb J}\otimes \theta _B)P_A
(I_{\mathbb J}\otimes S_B^{-1})P_A(I_{\mathbb J}\otimes \theta _B^*).$
\end{proposition}

\begin{proof}
\item[(i)] Let $\{e_j\}_{\j*}$ and $\{f_m\}_{m\in \mathbb M}$ denote the standard orthonormal bases of
 $\ell^2(\mathbb J)$ and  $\ell^2(\mathbb M)$, respectively. Then for any $x\in H$,
\begin{align*}
 (I_{\mathbb J} \otimes \theta _B)\theta _A (x) 
&= (I_{\mathbb J} \otimes \theta _B)\sum_{\j*} (e_j\otimes A_jx)
= \sum_{\j*} (e_j\otimes \theta _B(A_jx))\\
&= \sum_{m\in \mathbb M, j\in \mathbb J}(e_j\otimes f_m\otimes B_mA_jx) 
= \theta _C (x). 
\end{align*}
The formula for $S_C$ now follows directly, as well as the Parseval case.
\item[(ii)] If $P_A$ commutes with  $(I_{\mathbb J}\otimes S_B)$, then
\[ 
P_A(I_{\mathbb J}\otimes S_B^{-1})P_A = (P_A(I_{\mathbb J}\otimes S_B)P_A)^{-1},
\]
hence
$S_C^{-1}=S_A^{-1}\theta _A^* (I_{\mathbb J}\otimes S_B^{-1})\theta _AS_A^{-1}$,
and thus
\begin{align*}
P_C&=\theta_CS_C^{-1}\theta_C^*
=(I_{\mathbb J} \otimes \theta _B)\theta _AS_A^{-1}\theta _A^* (I_{\mathbb J}\otimes S_B^{-1})\theta _A
S_A^{-1}\theta _A^*(I_{\mathbb J} \otimes \theta _B)\\
&=(I_{\mathbb J} \otimes \theta _B)P_A (I_{\mathbb J}\otimes S_B^{-1})P_A(I_{\mathbb J} \otimes \theta _B).
\end{align*}

\end{proof}

\begin{remark}\label{R:compos} 
\item[(i)] If the composition of an \o* $\{B_m\}_{m\in \mathbb I}$ with the frame $\{A_j\}_{\j*} $
 is the same as the composition with the frame $\{A'_j\}_{\j*}$, then $\{A_j\}_{\j*} =
\{A'_j\}_{\j*} $. Indeed, if $B_mA_j=B_mA'_j$ for all $m\in \mathbb M$ and $j\in \mathbb J$, 
 then 
\[
\sum_{m\in M} B_m^*B_mA_j=S_BA_j=S_BA'_j,
\]
 and hence  $A_j=A_j'$, since $S_B$ is invertible.
\item[(ii)] Assume that $\bigcup _{j\in \mathbb J} A_j H_o$ is dense in $H_o$. Then the composition of $\{A_j\}_{j\in \mathbb J} $ with 
$\{B_m\}_{m\in \mathbb M}$ equals the composition of $\{A_j\}_{j\in \mathbb J} $ with $\{B'_m\}_{m\in \mathbb M}$
 if and only if $\{B_m\}_{i\in \mathbb I}=\{B'_m\}_{m\in \mathbb M}$.
\end{remark}

Given an \o*  $\{A_j\}_{j\in \mathbb J}$ with $A_j \in B(H,H_o)$ and given an arbitrary vector frame
$\{B_m\}_{m\in \mathbb M}$ with $B_m \in B(H_o, \mathbb C)$, we can view the vector frame
$\{C_{(m,j)} := B_mA_j\}_{(m,j)\in \mathbb M\times \mathbb J}$ to be a ``decomposition" of the original
\o*. 

This decomposition is of course not unique. For instance, in the discussion after Definition \ref{D:oper frame},
 we chose $\{B_m\}_{m\in \mathbb M}$ to be an orthonormal basis $\{e_m\}_{m\in \mathbb M}$ 
of $H_o$  and then $C_{(m,j)} := B_mA_j$ corresponds to the vector frame 
$\{A_j^*e_m\}_{(m,j)\in \mathbb M\times \mathbb J}$.
If $\{{B_m'}\}_{m\in \mathbb M}$ is (right) similar to $\{B_m\}_{m\in \mathbb M}$, i.e., ${B_m'} = B_mR$ 
for some invertible operator $R\in B(H_o)$  and all $m\in \mathbb M$ (i.e., if $B_m' $ is a Riesz frame 
(see Remark \ref{R:Riesz}), then $C'_{(m,j)} := B_mRA_j$ corresponds to the vector frame 
$\{A_j^*R^*e_m\}_{m\in \mathbb M, j\in \mathbb J}$ and provides a vector frame decomposition 
both of the \o*  $\{A_j\}_{j\in \mathbb J}$ and also of the left-similar  \o*  $\{RA_j\}_{j\in \mathbb J}$.
$\{A_j^*R^*e_m\}_{(m,j)\in \mathbb M\times \mathbb J}$ is similar to $\{A_j^*e_m\}_{(m,j)\in \mathbb M\times \mathbb J}$
only if $R$ satisfies the conditions of Proposition {P:left right}. However, if $T\in B(H)$ is invertible, 
$\{TA_j^*e_m\}_{(m,j)\in \mathbb M\times \mathbb J}$ is by definition similar to 
$\{A_j^*e_m\}_{(m,j)\in \mathbb M\times \mathbb J}$ and provides a vector decomposition of the \o* 
$\{A_jT^*\}_{j\in \mathbb J}$.

\section{Dual Frames, Complementary Frames, and Disjointness} \label{S:dual}
\subsection*{Dual frames}

A vector frame  $\{b_j\}_{j\in \mathbb J}$ on a Hilbert space $H$ is said to be the dual of another vector frame
 $\{a_j\}_{j\in \mathbb J}$ on $H$ if 

\[
x=\sum_{j\in \mathbb J}~(x,a_j)b_j \quad \text{for all}~ x\in H.
\]

It is well known and easy to see that this condition can be reformulated in terms of the analysis operators 
$\theta_b$ and $\theta_a$ of the two frames as
\[
\theta_b^*\theta_a = I
\]
In particular, $\{D_a^{-1}a_j\}_{j\in \mathbb J}$ has analysis operator $\theta_aD_a^{-1}$ and is called the \emph{canonical 
dual} of the frame $\{a_j\}_{j\in \mathbb J}$ (other duals are called \emph{alternate duals}).
When $\{a_j\}_{j\in \mathbb J}$ is Parseval, the identity
\[
x = \theta_a^*\theta_a x= \sum_{j\in \mathbb J}~(x,a_j)a_j \quad \text{for all}~ x\in H
\]
is called the \emph{reconstruction formula} (see \cite [Sections 1.2, 1.3]{HL00})

These notions extend naturally to operator valued frames:

\begin{definition}\label{D:dual} Let  $\{A_j\}_{j\in \mathbb
J},\{B_j\}_{j\in \mathbb J}\in \mathscr F$. Then 
$\{B_j\}_{j\in \mathbb J}$ is called a dual of
$\{A_j\}_{j\in \mathbb J}~$ if $\theta_B^*\theta_A =I$.  The operator-valued frame
$\{A_jS_A^{-1}\}_{j\in \mathbb J}$ is called the \emph{canonical
dual frame} of $\{A_j\}_{j\in \mathbb J}.$
\end{definition}

\bR{R:dual} \item[(i)] Notice that $\theta_B^*\theta_A =I$ if and only if $\theta_A^*\theta_B =I$, i.e., 
$\{B_j\}_{j\in \mathbb J~}$ is
 a dual of $\{A_j\}_{j\in \mathbb J}$ if and only if $\{A_j\}_{j\in \mathbb J}$ is
 a dual of $\{B_j\}_{j\in \mathbb J}$ (cfr. \cite[Proposition 1.13]{HL00}). 
\item[(ii)] If we consider two \o*s $\{A_j\}_{j\in \mathbb J}$ with $A_j\in B(H_1, H_o)$ and  $\{B_j\}_{j\in \mathbb J}$ 
with $B_j\in B(H_2, H_o)$ with a given unitary $U: H_1 \mapsto H_2$, and we want to keep track of the different
 Hilbert spaces, then we would define duality between them (relative to the choice of the unitary $U$) by asking that 
$\theta_B^*\theta_A =U$, or, equivalently, that $\theta_A^*\theta_B =U^*.$

\eR

In general, duals are far from unique. Theorem \ref {T:param} provides the natural way to parametrize the collection
 of all the dual frames of a given \o*. Recall that for  $\{A_j\}_{j\in \mathbb J},\{B_j\}_{j\in \mathbb J}\in \mathscr F$,
there is a unique $M\in B(\ell(\mathbb J)\otimes H_o)$ such that $\theta_B=M\theta_A$, $M=MP_A$ and
$M^*M$ is invertible in $B(P_A\ell(\mathbb J)\otimes H_o)$, namely $M:=\Phi_A(\{B_j\}_{j\in \mathbb J})$. In the two-by-two matrix relative to the decomposition\linebreak
$P_A+P_A^\bot=I\otimes I_o$,
$ M=\begin{pmatrix}
X & 0\\
Y & 0
\end{pmatrix}$ and $X^*X+Y^*Y$ invertible. 

\bP{P:dualparam}
Let $\{A_j\}_{j\in \mathbb J},\{B_j\}_{j\in \mathbb J}\in \mathscr F$ and let $M:=\Phi_A(\{B_j\}_{j\in \mathbb J})$. 
Then the following conditions are equivalent.
\item [(i)] $\{B_j\}_{j\in \mathbb J}$ is a dual of $\{A_j\}_{j\in \mathbb J}$.
\item [(ii)] $\sum_{j\in \mathbb J}~B_j^*A_j=I$ where the convergence is in the strong operator topology.
\item [(iii)] $P_AMP_A= \theta_A^*S_A^{-2}\theta_A$.

\noindent In particular, $\{B_j\}_{j\in \mathbb J}$ is the canonical dual of $\{A_j\}_{j\in \mathbb J}$ if and only 
if $M=P_AMP_A$.
\eP
\bp
\item [(i)] $\Longleftrightarrow$ (ii). Obvious by (\ref{e:2transf})
\item [(i)] $\Longleftrightarrow$ (iii) By Theorem \ref{T:param}, there is a unique operator $M\in B(\ell(\mathbb J)\otimes H_o\,)$ with 
$M=MP_A$ and $\left.M^*M P_A\right.|_{P_A\negmedspace \ell(\mathbb J)\otimes H_o}$ invertible,  for which $B_j=L_j^*M\theta_A$ for all 
$j\in \mathbb J$, or, equivalently, for which $\theta_B= M\theta_A$. Then 
$I= \theta_A^*\theta_B = \theta_A^*M\theta_A $ if and only if 
\[
P_AMP_A= \theta_A^*S_A^{-1}\theta_A^*M\theta_AS_A^{-1}\theta_A^*=\theta_A^*S_A^{-2}\theta_A.
\]
If for some $M=MP_A$, $P_AMP_A= \theta_A^*S_A^{-2}\theta_A$ which is invertible, then also 
$\left.M^*M P_A\right.|_{P_A\negmedspace \ell(\mathbb J)\otimes H_o}$ invertible.

By definition, $\{B_j\}_{j\in \mathbb J}~$ is the canonical dual of $\{A_j\}_{j\in \mathbb J}~$ if and only if
\[
\theta_B= \theta_AS_A^{-1} = \bigl( \theta_AS_A^{-2}\theta_A^*\bigr)\theta_A,
\]
i.e., if and only if $M= \theta_AS_A^{-2}\theta_A^*$.

\ep

By the above proposition, the only \o*s that have a unique dual frame are those with range projection 
$P_A = I \otimes I_o$. By Remark \ref {R:Riesz}, these are the Riesz \o*s (cfr. \cite [Corollary 2.26]{HL00}).

Given an \o* $\{A_j\}_{j\in \mathbb J}$, set $M:=\Phi_A(\{B_j\}_{j\in \mathbb J})=\begin{pmatrix}
X & 0\\
Y & 0
\end{pmatrix}$.
Proposition \ref{P:dualparam} (iii) states that dual frames are characterized by $X= \theta_AS_A^{-2}\theta_A^*$ 
(by $X=P_A$ if $\{A_j\}_{j\in \mathbb J}$  is Parseval.)  

Frames that are right-similar to $\{A_j\}_{j\in \mathbb J}$ are characterized by $Y=0$ (see Proposition \ref {P:simil}).
Thus  the only frame that is both right-similar and dual to $\{A_j\}_{j\in \mathbb J}$ corresponds to
$X= \theta_AS_A^{-2}\theta_A^*$ and $Y=0$ and hence it is the canonical dual of 
$\{A_j\}_{j\in \mathbb J}$.  Equivalently, if two right-similar frames are both the dual
 of a given frame, then they are equal. To summarize:

\bP{P:LH} \item[(i)] (cfr. \cite[Proposition 1.14]{HL00}) If two \o*s are right-similar and are dual of the same
\o*, then they are equal.
 \item[(ii)] If two Parseval \o*s are one the dual of the other, then they are equal.
\eP

In general there are infinitely many dual frames of a given \o* that are left-similar to it.
\begin{example}\label{E:Sdual}. Let $H:=\ell^2(\mathbb J)$, $H_o:=\mathbb C^2$,  
$Q_1:=\begin{pmatrix}
1 & 0\\
0 & 0
\end{pmatrix},
Q_2:=\begin{pmatrix}
0 & 0\\
0 & 1
\end{pmatrix}$, 
  and  
$R:=\begin{pmatrix}
1 & 0\\
\lambda & 1
\end{pmatrix}$. 
Define $P:=I\otimes Q_1$ and the \o*s   \linebreak
$\{A_j:=L_j^*\left. P\right|_{P \ell(\mathbb J)\otimes H_o}\}_{j\in \mathbb J}$, and $\{B_j:=RA_j\}_{j\in \mathbb J}$.
Then by Example \ref{E:dil}, $\{A_j\}$ is a Parseval operator-valued frame  and $P_A = P$. 
Now $M=(I\otimes R)P$, $PMP = P$ and hence $\{B_j:=RA_j\}_{j\in \mathbb J}$ is a dual of  
$\{A_j\}_{j\in \mathbb J}$ for every $\lambda$. However, $P^\bot MP= I\otimes\begin{pmatrix}
0 & 0\\
\lambda & 1
\end{pmatrix}$, hence every $\lambda$ defines a different \o* $\{B_j\}_{j\in \mathbb J}$.

\end{example}

\subsection*{Disjoint frames and complementary frames}
The treatment given in \cite [Chapter 2] {HL00} for the vector case, generalizes without difficulties to the
higher multiplicity case. For the readers' convenience we present the definitions and one of the key arguments.

\bD{D:complem}
Let $\{A_j\}_{j\in \mathbb J}$ and $\{B_j\}_{j\in \mathbb J}$ be two \o*s with \linebreak $A_j \in B(H_A,H_o)$ and 
$B_j \in B(H_B,H_o)$  for all $j\in \mathbb J$ .
Then the two frames are called:
\item [(i)] disjoint if  $\{A_j \oplus B_j\}_{j\in \mathbb J}$ is an \o*.
\item [(ii)] strongly disjoint (also called orthogonal) if there are two invertible operators $T_A\in B(H_A)$ and \\
$T_B\in B(H_B)$
such that $\{A_jT_A \oplus B_jT_B\}_{j\in \mathbb J}$ is a Parseval \o* on $H_A \oplus H_B$
\item [(iii)] strongly complementary if there are two invertible operators $T_A\in B(H_A)$ and $T_B\in B(H_B)$
such that $\{A_jT_A \oplus B_jT_B\}_{j\in \mathbb J}$ is an orthonormal \o* on $H_A \oplus H_B$.
\eD

\bP{P:complem}
Let $\{A_j\}_{j\in \mathbb J}$ and $\{B_j\}_{j\in \mathbb J}$ be two \o*s as in the Definition \ref {D:complem}. Then
the two frames are:

\item [(i)]disjoint if and only if $P_AH_{\mathbb J~}\cap P_BH_{\mathbb J~} = \{0\}$ and 
$P_AH_{\mathbb J~}+P_BH_{\mathbb J~}$ is closed
\item [(ii)] strongly disjoint (orthogonal) if and only if $P_AP_B=0$ (i.e., $P_A\bot P_B$) if and only if $\theta_A^*\theta_B=0$,
 if and only if $\theta_B^*\theta_A=0$
\item [(iii)] strongly complementary if and only if $P_A + P_B=I\otimes I_o.$
\eP

\bp
\item [(i)] $\{A_j \oplus B_j\}_{j\in \mathbb J}$ is an \o* if and only if $\{A_jT_A \oplus B_jT_B\}_{j\in \mathbb J}$
is also an \o* for any choice of invertible operators $T_A$ and $T_B$. Since right similarities do not change the 
frame projections (see Proposition \ref {P:simil}), we can assume that both frames are Parseval.
Since $(A_j\oplus B_j)(x\oplus y) = A_jx+ B_jy$ for all $x\in H_A,~y\in H_B$, a simple computation shows
 that $(A_j\oplus B_j)^*z= A_j^*z\oplus B_j^*z$ for all $z\in H_o$, hence
\begin{align*}
\bigl((A_j\oplus B_j)&^*(A_j\oplus B_j)(x\oplus y), (x\oplus y) \bigr)= (A_jx + B_jy,A_jx+ B_jy)\\
&=(A_jx,A_jx) +  (B_jy,B_jy ) +(A_jx,B_jy)+(B_jy,A_jx)\\
&= (A_j^*A_jx,x) + (B_j^*B_jy, y)+(B_j^*A_jx, y)+(A_j^*B_jy, x).
\end{align*}
Now sum over  $\mathbb J$ using the fact (see (\ref{e:2transf})) that all the
series converge in the strong operator topology, e.g., $\sum_{j\in \mathbb J}A_j^*B_j=\theta_A^*\theta_B$ and
similarly for the other series. Thus 
\begin{align*}
\bigl(\sum_{j\in \mathbb J} \bigl((A_j\oplus B_j)^*(A_j\oplus B_j)&(x\oplus y), (x\oplus y) \bigr) = 
(\theta_A^*\theta_Ax,x)+(\theta_B^*\theta_By,y)+ \\
&+( \theta_B^*\theta_Ax, y) +( \theta_A^*\theta_By, x) = \bigl(\theta_Ax+\theta_By,\theta_Ax+\theta_By \bigr).
\end{align*}
As in the proof of \cite [Theorem 2.9]{HL00}, the projections $P_A$ and $P_B$ satisfy the condition in (i) if and only if
\[
 a\|\theta_Ax+\theta_By\|\le \sqrt{\|\theta_Ax\|^2+\|\theta_By\|^2} = \sqrt{\|x\|^2+\|y\|^2} = \|x\oplus y\|
\le b\|\theta_Ax+\theta_By\|
\]
for some $a, b > 0$. By definition, this is precisely the condition that guarantess that 
$\{A_j \oplus B_j\}_{j\in \mathbb J}$ is an \o*.

\item[(ii)] It is clear that the three conditions $P_AP_B=0$, $\theta_A^*\theta_B=0$, and 
$\theta_B^*\theta_A=0$ are all equivalent and that they also imply the condition in (i).  If these conditions hold, 
by the proof of (i),
$\bigl((D_{A\oplus B}(x\oplus y),(x\oplus y)\bigr)=(S_Ax,x)+(S_By,y)$ for all $x\in H_A,~y\in H_B$. In particular,  
$D_{A\oplus B}= I_A\otimes I_B$ if (and only if) $S_A=I_A$ and $S_B=I_B$, i.e., it is sufficient to choose
 right similarities that make the two \o*s Parseval to obtain that their direct sum is also Parseval.
Conversely, assume without loss of generality that the direct sum of the frames is already Parseval and hence 
both frames are Parseval, then,
 again by the computation in the first part of the proof, 
\[
\|\theta_A x\|^2+\|\theta_Ay\|^2=\|x\|^2+\|y\|^2=\bigl(D_{A\oplus B}(x\oplus y),(x\oplus y)\bigr) 
= \|\theta_A x+ \theta_B y\|^2.
\]
This clearly implies that the ranges of $ \theta_A$ and $ \theta_B$ are orthogonal

\item[(iii)] From the proof of (ii), it is easy to see that if $P_AP_B=0$, then $P_{A+B}=P_A+P_B$. The 
rest of the proof is then obvious.

\ep

\bC{C:param complem} Let $\{A_j\}_{j\in \mathbb J}$ be an \o* on a Hilbert space $H$, with range in $H_o$. 
\item[(i)] Up to right unitary equivalence, the collection of strong complements of  $\{A_j\}_{j\in \mathbb J}$
 is uniquely parametrized by 
\[
\bigl\{\{L_j^*T \}_{\j*}\mid  T=P_A^\bot TP_A^\bot \ge 0,\, T \text{ invertible in }  B(P_A^\bot \ell(\mathbb J)\otimes H_o)\bigr\}.
\]
\item[(ii)] Up to right unitary equivalence, the collection of \o* strongly disjoint from  $\{A_j\}_{\j*}$
  is uniquely parametrized by
\[
\bigl\{\{L_j^*T \}_{\j*} \mid P\le P_A^\bot, T=PTP \ge 0,\, T \text{ invertible in } 
 B(P\ell(\mathbb J)\otimes H_o) \bigr\}.
\]
\eC

\section{Unitary Systems and Groups} \label{S:groups}
\subsection*{General Unitary Systems and Local Commutants}

Following the terminology of \cite {DL98}  and \cite {HL00}  a unitary system $\mathscr U$ on a Hilbert space $H$ is simply a collection of unitary operators that includes the identity. Following the customary terminology for (vector) frames, we introduce the analogous notion for operator-valued frames:

\begin{definition}\label{D:generator} An operator $A\in B(H,H_o)$ is called an operator frame generator 
for a unitary system $\mathscr U$ if $\{AU^*\}_{U\in \mathscr U}$ is an \o*. If $\{AU^*\}_{U\in \mathscr U}$
is Parseval, $A$ is said to be a Parseval operator frame generator.
\end{definition}

If $\dim H_o = 1$, i.e., $A$ corresponds to a (unique) vector $\psi$, then $AU^*$ corresponds to the vector $U\psi$.
Recall from \cite[Proposition 3.1]{HL00} that given a wavelet generator for a unitary
 system $\mathscr U$, i.e., a vector $\psi$ such that $\{U\psi\}_{U\in \mathscr U}$ is an orthonormal basis, 
then a vector $\phi$ is a Parseval frame generator from $\mathscr U$ 
if and only if $\phi=V\psi$ for a (unique) co-isometry $V$ such that $(VU-UV)\psi=0$ for every $U\in\mathscr U$. 
The \emph{local commutant} at $\psi$ is defined in \cite{DL98} as the collection  
\[
C_{\psi}(\mathscr U) := \{T\in B(H) \mid (TU-UT)\psi=0 \quad\text{for all} \quad U\in \mathscr U \}. 
\] 
The multi-dimensional analog of an orthonormal basis is a collection of partial
 isometries with mutual orthogonal equivalent domains spanning the Hilbert space, or equivalently, an \o* with 
frame projection $\theta _A\theta _A^*=I$.

\begin{proposition}\label{P:comm} Suppose that $A\in B(H,H_o)$ is a frame generator 
for a unitary system $\mathscr U$ for which $\theta_A\theta_A^*=I$. Then an operator $B\in B(H,H_o)$ is
a Parseval frame generator for $\mathscr U$ if and only if $B=AV^*$ for some co-isometry $V$
 such that $(VU-UV)A^*=0$ for every $U\in\mathscr U$. If $B=AV^*$ for such co-isometry, then 
$V:= \theta_B^*\theta_A$.  
\end{proposition}

\begin{proof} Assume $B$ is a Parseval frame generator for $\mathscr U$ and let $V:= \theta_B^*\theta_A$. Then
 by the hypotheses, $\theta_B^*\theta_A\theta_A^*\theta_B=I$, i.e., $V$ is a  co-isometry and
\[
B= BI=L_{I}\theta_B= L_{I}\theta_A \theta_A^*\theta_B = AV^*.
\]
Since and $AU^* = L_U^*\theta_A$ and $BU^* = L_U^*\theta_BA$ for all $U\in \mathscr U$, we have
\[
AU^*V^* = L_U^*\theta_A\theta_A^*\theta_B =L_U^*\theta_B = BU^* = AV^*U^*.
\]
Moreover, if $AV^*=AW^*$ and $A(WU-UW)^*=0$ for every $U\in\mathscr U$, then also
$AU^*V^*=AU^*W^*$ and hence $\theta_A V^* = \theta_AW^*$ whence $V=W$.
Conversely, if $V$ is co-isometry and $A(VU-UV)^*=0$ for every $U\in\mathscr U$, then 
\[
\sum_{U\in \mathscr U}UVA^*AV^*U^* = V\left (\sum_{U\in \mathscr U}UA^*AU^*\right)V^*=VV^*=I
\]
whence $AV^*$ is a Parseval frame generator  for $\mathscr U$.

\end{proof}

\subsection*{Discrete Group Representations}

Let $G$ be a discrete group, not necessarily countable and let  $\lambda$  be the left
 regular representation of $G$ (resp., $\rho$ the right regular representation of $G$). Denote by 
$\mathscr L (G)\subset B(\ell^2(G))$,
 (resp., $\mathscr R (G))$  be the von Neumann algebra generated by 
the unitaries  $\{\lambda _g\}_{ g\in G}$, (resp., $\{\rho _g\}_{ g\in G}$). It is well known that
both  $\mathscr L (G)'=\mathscr R (G)$ and $\mathscr R (G)' = \mathscr L (G)$  are finite von Neumann algebras
that share a faithful trace vector $\chi_e$, where
$\{\chi_g\}_{g\in G}$ is the standard basis of $\ell^2(G)$. 

Let $ H_o$ be a Hilbert space and $I_o$ be the identity of $B(H_o)$. Then we call
$\lambda \otimes id\,: G \ni g\,\longrightarrow \lambda_g\otimes I_o$ the left regular
 representation of $G$ with multiplicity $H_o$. Set   $H_G=\ell^2(G) \otimes \ H_o$.

Given a unitary representation $ (G, \pi,   H)$, denote by $\pi(G)''$ the von Neumann algebra
generated by $\{\pi_g\}_{g\in G}$. Operator frame generators, if any, 
for the unitary system $\{\pi_g\}_{g\in G}$, are called generators for the representation. Explicitly:

\begin{definition}\label{D:repr gen} Let $(G, \pi,   H)$ be a unitary representation of the discrete
 group $G$ on the Hilbert space $H$. Then an operator $A\in B( H,  H_o)$ is called a frame generator 
(resp. a Parseval frame generator) with range in $H_o$  for the representation 
if  $\{A_g:=A\pi_{g^{-1}}\}_{g\in G}$ is an \o* (resp. a Parseval \o*).
\end{definition}

Before characterizing those representations that have an operator frame generator and then parametrizing its
 generators, we need the following preliminary lemma.

\begin{lemma}\label{L:basic group} Let $A $ and $ B$ be two  generators with range in $H_o$ for a unitary representation $(G, \pi,   H)$. Then 
 \item [(i)] $\theta_A\pi_g= (\lambda _g\otimes I_o)\theta _A$ for all $g\in G$.
\item [(ii)] $\theta_A^* \theta_B$ is in the commutant $\pi(G)'$ of $\pi(G)"$. In particular,
$S_A\in \pi(G)'$ and $AS_A^{-1/2}$ is a Parseval frame generator.
\item [(iii)] $\theta_A T\theta_B^*\in \mathscr R (G)\otimes B(H_o)$ for any $T\in \pi(G)'$. In particular, $P_A\in \mathscr R (G) \otimes B(H_o)$.
\item [(iv)] $P_A\sim P_B$, where the equivalence is in $\mathscr R (G)\otimes B(H_o)$, i.e., 
it is implemented by a partial isometry belonging to $\mathscr R (G)\otimes B(H_o)$.
\end{lemma}
\begin{proof}
(i) For all $g,\,q\in G$ and $h\in H_o$, one has 
\[
L_{gq}h = \chi_{gq}\otimes h = \lambda_g\,\chi_q\otimes h = (\lambda_g\otimes I_o)\,(\chi_q \otimes h)
= (\lambda_g\otimes I_o)\, L_qh.
\]
Thus
\begin{align*}
\theta_A \pi_g
&= \sum_{p\in G}L_pA\pi_{p^{-1}}\pi_g
= \sum_{p\in G}L_pA\pi_{p^{-1}g}
=\sum_{q\in G}L_{gq}A\pi_{q^{-1}}\\
&= \sum_{q\in G}(\lambda _g \otimes I_o)L_qA\pi_{q^{-1}} 
= (\lambda _g\otimes I_o) \theta _A.
\end{align*}

(ii) For all $g\in G$ one can  apply (i) twice and obtain
\[
\theta_A^*\theta_B \pi_g=  \theta_A^*(\lambda _g\otimes I_o)\theta _B =\pi_g \theta_A^*\theta_B.
\]
Thus, $\theta_A^*\theta_B \in \pi(G) '$. In particular, setting $B=A$ we have 
 $S_A =\theta_A^*\theta_A \in \pi(G)'$. Then  $AS_A^{-1/2}\pi_{g^{-1}} = A\pi_{g^{-1}}S_A^{-1/2}$ 
for all $g\in G$, $AS_A^{-1/2}$  is a Parseval frame generator.

(iii) For all $g\in G$ and $T\in \pi(G)'$ applying twice (i), one obtains
\[
\theta_A T\theta_B^*(\lambda_g\otimes I_o)
=\theta_A T \pi_g\theta_B^*
=\theta_A  \pi_g T\theta_B^*
=(\lambda_g\otimes I_o)\theta_A  T\theta_B^*.
\]
Therefore, 
\[
\theta_A T\theta_B^* \in (\mathscr L (G)\otimes I_o)' = \mathscr L (G)'\otimes (I_o)'= \mathscr R (G) \otimes B(H_o).
\] 
Setting $A=B$ and $T = S_A^{-1}$, we see that $P_A = \theta_A S_A^{-1}\theta_A^*\in \mathscr R (G) \otimes B(H_o)$. 

(iv) By passing if necessary to  $AS_A^{-1/2}$ (resp., $BS_B^{-1/2}$) which by (ii) is Parseval frame generator
by Proposition \ref{P:simil} has frame projection $P_A$ (resp., $P_B$),
we can assume, without loss of generality, that both $\theta _A$ and $\theta _B$ are isometries. Then the
partial isometry 
$V=\theta_B\theta_A^* \in \mathscr R (G) \otimes B(H_o)$ implements the equivalence, i.e., 
$V^*V=P_A$ and $VV^*=P_B$.
\end{proof}
 
Given a countable group, in \cite[Theorems 3.8, 3.11, Proposition 6.2] {HL00}, Han and Larson have identified  
its representation that  have a multi-frame (vector) generator with the subrepresentations of its left regular  representation with finite multiplicity. Lemma \ref{L:basic group} permits to easily reobtain 
their result with a slight increase in generality.

\begin{theorem}\label{T:repres} A unitary representation $(G, \pi, H)$ of a discrete group 
is unitarily equivalent to a subrepresentation of $\lambda \otimes id $ with multiplicity $H_o$ if and only if  $(G, \pi, H)$  has an operator frame generator with range in $H_o$. 
\end{theorem}
\begin{proof} 
Assume that $(G, \pi, H)$ is unitarily equivalent to, and hence can be identified with, a subrepresentation
of $\lambda\otimes id$, the restriction $\lambda\otimes id\left. \right |_{PH_G} $ 
 for some projection $P\in (\mathscr L (G)\otimes I_o)'= \mathscr R (G)\otimes B(H_o)$.  
Let $H:=PH_G$, $A:= \left. L_e^*P\right |_H$, and 
\[
A_g=\left. L_e^*P\right |_H(\lambda _g\otimes I_o)\left. P\right |_H 
= L_e^*(\lambda _g\otimes I_o)\left. P\right |_H \quad \text{for all $g\in G$}.
\]
 Then
\begin{align*}
S_A&=\sum _{g\in G}A_g^*A_g
=\sum _{g\in G}P(\lambda _g\otimes I_o)L_eL_e^*(\lambda _{g^{-1}}\otimes I_o)\left. P\right |_H\\
&=P(\sum _{g\in G}L_gL_g^*)\left. P\right |_H
 =\left. I\right |_H.
\end{align*}
This shows that $A$ is a (Parseval) frame generator for $(G, \pi, H)$.

Conversely, assume that $A\in B(H, H_o)$ is a frame generator for 
$(G, \pi, H)$. Then $\theta _AS_A^{-1/2}$ is an isometry onto the subspace $P_AH_G$ and
 $S_A^{-1/2}$ commutes with $\pi$ and 
$P_A$ commutes with the left regular representation with multiplicity $H_o$  by 
Lemma \ref{L:basic group} (ii) and (iii). 
Then by Lemma \ref{L:basic group} (i) 
\[ 
\theta_AS_A^{-1/2}\pi_g=\theta_A\pi_gS_A^{-1/2}= (\lambda _g\otimes I_o)\theta_AS_A^{-1/2}=
(\lambda _g\otimes I_o)\left. P_A\right |_{P_A\negmedspace H_G}\theta _AS_A^{-1/2}
\]
for all $g\in G$, i.e., $\pi$ is unitarily equivalent to $\lambda \otimes id\left. P_A\right |_{P_A\negmedspace H_G}.$
\end{proof}

From the above proof it is easy to obtain the following:
 
\begin{remark}\label{R:equiv}
If $A\in B(H, H_o)$ is a frame generator for $(G, \pi, H)$, then
the equivalence of  $(G, \pi, H)$ and $(G, \lambda \otimes id \left.\right |_{P_A\negmedspace H_G}\negthinspace, P_AH_G)$ 
is implemented by the isometry  $\theta_AS_A^{-1/2}$. An isometry $V$ implements this equivalence if and only if
$V=\theta_AS_A^{-1/2}U$ for some unitary operator $U\in \pi(G)'$.
\end{remark}

It is well known and easy to see that two subrepresentations of the left regular representation with 
multiplicity  $H_o$, $(\lambda \otimes id)\left. P\right |_{PH_G}$ and 
$(\lambda \otimes id)\left. Q\right |_{QH_G}$,  
are equivalent if and only if $P\sim Q$ in $\mathscr R (G)\otimes B(H_o)$.
In other words, the equivalence classes of subrepresentations of the left regular representation 
with fixed multiplicity  $H_o$ are identified with the collection of equivalence classes of projections of the von Neumann algebra $\mathscr R (G)\otimes B(H_o)$.

Theorem \ref {T:repres} permits to characterize those operator valued frames labeled by a discrete group $G$ that have a frame generator. For simplicity's sake, because of  Remark \ref {R:uniq} we need to consider only Parseval \o*s.

\begin{proposition}\label{P: have generator}
Let $G$ be a discrete group and let $\{A_g\}_{g\in G}$ be a Parseval \o* in $B(H,H_o)$. Then there is a unitary representation $\pi$ of $G$ on $H$ for which $A_g =A_e\pi_{g^{-1}}$ for all $ g\in G$ if and only if $A_{gp}A^*_{gq} = A^*_pA_q$ for all $p,q,g \in G$.
\end{proposition}
\bp
Assume  $A_g =A_e\pi_{g^{-1}}$ for some unitary representation $\pi$ and for all $ g\in G$. Then 
\[
A_{gp}A^*_{gq} =A_e\pi_{(gp)^{-1}}(\pi_{(gp)^{-1}})^*A_e^*
=A_e\pi_{p^{-1}}\pi_{g^{-1}}\pi_g(\pi_q^{-1})^*A_e^*
= A_p^*A_q
\]
for all $p,q,g \in G$. Assume now that  $A_{gp}A^*_{gq} = A^*_pA_q$. Then for every $g \in G$, by the proof of Lemma \ref {L:basic group} (i),

\begin{align}
(\lambda _g\otimes I_o)&\theta _A\theta _A^*(\lambda _g\otimes I_o)^*
=\sum_{p,q\in G}(\lambda _g\otimes I_o)L_pA_pA_q^*L_q^*(\lambda _g\otimes I_o)^*\\
&=\sum_{p,q\in G}L_{gp}A_pA_q^*L_{gq}
=\sum_{r,s\in G}L_rA_{g^{-1}r}A_{g^{-1}s}^*L_s\\
&=\sum_{r,s\in G}L_rA_{r}A_{s}^*L_s
=\theta _A\theta _A^*.
\end{align}
This proves that the projection $P_A=\theta _A\theta _A^* \in \mathscr R (G)\otimes B(H_o).$ But then  the operator valued weight $\{A_g\}_{g\in G}$ can be identified to the compression to $P_A$ of the left regular representation $\lambda _g\otimes I_o$ which has an operator frame generator. Explicitly, again by the proof of Lemma \ref {L:basic group} (i),
\[
A_g = L_g^*\theta_A 
=L_e^*P_A(\lambda _{g^{-1}}\otimes I_o)P_A\theta_A =A_e\theta_A^*(\lambda _{g^{-1}}\otimes I_o)P_A\theta_A.
\]
Since $(\lambda _g\otimes I_o)P_A$ is a unitary representation of $G$ on the Hilbert space $P_A$, then
$\pi_g:= \theta_A^*(\lambda_g\otimes I_o)P_A\theta_A$ is a unitary representation of $G$ on the 
Hilbert space $H$. Thus $A_g:=A_e\pi_{g^{-1}}$, i.e., $A_e$ is a frame generator for $(G,\pi,H)$
\ep

\bR {R: special case}
\item[(i)]  Proposition \ref{P: have generator}  is a generalization of the following known result
for group-indexed frames: When $\dim H_o =1$, i.e., in the case of a Parseval vector frame  $\{x_g\}_{g\in G}$, the necessary and sufficient condition for that frame to have a generator for some unitary representation of $G$ (necessarily unitarily equivalent to a subrepresentation of the left regular representation) is that
\begin{equation}\label{e:group}
<x_{gp}, x_{gq} >= <x_p, x_q > \quad\text{for all} \quad p,q,g \in G. 
\end{equation}
This result  can be deduced from the material in Chapter 3 of \cite {HL00}, however it was not stated
explicitely in that paper. Condition (\ref {e:group}) is clearly equivalent to
the condition that the range of the analysis operator is invariant under the left regular representation of the group on the analysis space, and so the frame can be obtained from the standard
orthonormal basis for this representation by simply projecting, thereby
obtaining  the required subrepresentation of G.  We note that Nga Nguyen has written an exposition of this in a forthcoming article stemming from her thesis research, along with some extensions to frames
satisfying this condition which are not-necessarily Parseval, where the situation is
more complicated.  We also note that some special cases, notably  where $G$ is
a cyclic group on a finite dimensional Hilbert space, have been independently proven and used by others. 
\item[(ii)] More is in true in a case where the group is abelian,  (see also Remark \ref {R:special case2}.)  If $G$ is abelian then Cor. 3.14 or Theorem 6.3 of  \cite {HL00} states that all group frames for a unitary representation of $G$ on the same Hilbert space are unitarily equivalent. So in the abelian case two frames both satisfying the condition (\ref {e:group})  are necessarily unitarily equivalent.

\item[(iii)] We do not know if there is a similar necessary and sufficient condition for frames indexed by a unitary system, or at least by some structured unitary system, like a Gabor system.
\eR

\section{Parametrization of operator frame generators} \label{S:param gen}

Theorem \ref{T:param} shows how to parametrize all \o*s with a given multiplicity in terms of a single \o*. 
This general result can be applied to parametrize all operator frame generators for a unitary  representation 
of a discrete group in terms if a single operator frame generator.

\begin{theorem}\label{T:param gen}
Let $A\in B(H,H_o)$ be a frame generator for the unitary representation  $(G, \pi, H)$. 
\item[(i)] If  $B(H_G)\ni M= MP_A$ and
$\left.M^*M\right|_{P_A\negmedspace H_G}$ is invertible in $B(P_AH_G)$, then $L_e^*M\theta_A$ is a frame generator for
 $(G, \pi, H)$ if and only if $M\in \mathscr R (G)\otimes B(H_o)$. 
\item[(ii)] The collection $\mathscr F_G$ of
all the operator frame generators for $(G, \pi, H)$ with the same multiplicity $H_o$  is 
uniquely parametrized as
\[
\mathscr F_G =\{L_e^*M\theta_A \mid  M\in \mathscr R (G)\otimes B(H_o), M=MP_A,\left.M^*M\right
|_{P_A\negmedspace H_G}\text{ is invertible}\}.
\]
\item[(iii)] If $A$ is Parseval, the collection of all the Parseval operator frame generators for 
$(G, \pi, H)$ with  multiplicity $H_o$  is uniquely parametrized as 
\begin{equation}\label{eq:Parseval}
\{L_e^*V\theta_A \mid  V\in \mathscr R (G)\otimes B(H_o), V^*V= P_A \}.
\end{equation}
If $B=L_e^*V\theta_A$ with $V\in \mathscr R (G)\otimes B(H_o)$ and $V^*V= P_A$, then $V=\theta_B\theta_A^*$.
\item[(iv)] If $A$ is Parseval and $P\sim P_A$ in $\mathscr R (G)\otimes B(H_o)$, then $P=P_B$ for 
$B=L_e^*V\theta_A$ and $V\in \mathscr R (G)\otimes B(H_o)$ with $V^*V= P_A$ and $VV^* = P$
\end{theorem}

\begin{proof} \item[(i)] If $M=MP_A$ is an operator in $\mathscr R (G)\otimes B(H_o)$ and $\left.M^*M\right
|_{P_A\negmedspace H_G}$ is invertible, by Theorem \ref{T:param} (ii), $\{L_g^*M\theta_A\}_{g\in G}$ is an \o*. 
But then for all $g\in G$
\[
L_g^*M\theta_A= (L_e^*\lambda_{g^{-1}}\otimes I_o) M\theta_A = 
L_e^*M\lambda_{g^{-1}}\otimes I_o\theta_A =
L_e^*M\theta_A\pi_{g^{-1}},
\] 
by  Lemma \ref{L:basic group} (i) i.e., $L_e^*M\theta_A$ is the generator of $\{L_g^*M \theta_A\}_{g\in G}$.
Conversely, assume that $L_e^*M\theta_A$ is an operator frame generator for $(G, \pi, H)$, i.e., that 
$\{L_e^*M\theta_A\pi_{g^{-1}}\}_{g\in G}$ is a frame. For all $g\in G$ we have
\[
L_e^*M\theta_A\pi_{g^{-1}}
= (L_e^*\lambda_{g^{-1}}\otimes I_o)(\lambda_g\otimes I_o)M(\lambda_{g^{-1}}\otimes I_o)\theta_A
=L_g^*(\lambda_g\otimes I_o)M(\lambda_{g^{-1}}\otimes I_o)\theta_A.
\]
Set $N_g:=(\lambda_g\otimes I_o)M(\lambda_{g^{-1}}\otimes I_o)P_A$. Then obviously $N_g=N_gP_A$ and
\begin{align*}
N_g^*N_g&= P_A(\lambda_g\otimes I_o)M^*(\lambda_{g^{-1}}
\otimes I_o)(\lambda_g\otimes I_o)M(\lambda_{g^{-1}}\otimes I_o)P_A\\
&=(\lambda_g\otimes I_o)P_AM^*MP_A(\lambda_{g^{-1}}\otimes I_o).
\end{align*} 
Since by hypothesis $P_AM^*MP_A$ is invertible in $B(P_AH_G)$ and since 
$\lambda_g\otimes I_o$ commutes with $P_A$, $N_g^*N_g$ is also invertible in $B(P_AH_G)$.
 But then by the uniqueness part of Theorem \ref{T:param}, 
$N_g=N_e= M$, i.e., $M$ commutes with $\lambda_g\otimes I_o$ for all $g\in G$ and hence 
$M\in \mathscr R (G)\otimes B(H_o)$.

\item[(ii)] If $B\in B(H,H_o)$ is an operator frame generator for $(G, \pi, H)$, then by Theorem \ref{T:param},  
$B\pi_{g^{-1}} =  L_g^*M\theta_A$ for some unique $ M=MP_A$ for which $\left.M^*M\right
|_{P_A\negmedspace H_G}$  is invertible. In particular, $B=  L_e^*M\theta_A$ and hence  
$M\in \mathscr R (G)\otimes B(H_o)$ by the above proof.
\item[(iii)] and (iv) The rest of the proof follows  by the same arguments and Corollary \ref{C:param Pars}.

\end{proof}

Special cases of operator frame generators arise from right or left similarities. We need first the following lemma

\begin{lemma}\label{L: 1} Let $A\in B(H,H_o)$ be an operator frame generator for $(G, \pi, H)$. 
\item[(i)] Let $R\in B(H_o)$ be invertible. Then $RA$ is an operator frame generator for $(G, \pi, H)$.
\item[(ii)] Let $T\in B(H)$ be invertible. Then $\{A\pi_{g^{-1}}T\}_{g\in G}$,  has a generator 
(necessarily $AT$) if and only if $T\in \pi(G)'$. 
\item[(iii)] Let $T\in B(H)$ be invertible. If $T\in\pi(G)"$, then $AT$ is an operator frame generator for $(G, \pi, H)$ and $AT= L_e^*(Y\otimes I_o)\theta_A$
for some invertible operator \\$Y \in \mathscr R (G)$. If $T$ is unitary, then $Y$ can be chosen to be unitary.
\end{lemma}
\begin{proof}
\item[(i)] Obvious.
\item[(ii)] The sufficiency of the condition is clear. For the necessity, assume that  $\{A\pi_{g^{-1}}T\}_{g\in G}$,  has a generator $B$, i.e., 
$A\pi_{g^{-1}}T = B\pi_{g^{-1}}$ for all $g \in G$. Then $\theta_AT= \theta_B$ by Lemma \ref{L:rightsim} and hence 
$T= S_A^{-1}\theta_A^*\theta_B \in \pi(G)'$ by Lemma \ref{L:basic group} (ii).
\item[(iii)] By Lemma \ref{L:basic group} (ii), $T$ commutes with $S_A^{-1/2}$, hence
 \[
AT = L_e^*\theta_AT=  L_e^*\theta_ATS_A^{-1}\theta_A^*\theta_A 
= L_e^*(\theta_AS_A^{-1/2}TS_A^{-1/2}\theta_A^*)\theta_A.
\]
By Lemma \ref{L:basic group} (i) and (ii),
\[
\theta_AS_A^{-1/2}\pi_gS_A^{-1/2}\theta_A^* = \theta_A\pi_gS_A^{-1}\theta_A^*
=(\lambda_g\otimes I_o) \theta_AS_A^{-1}\theta_A^*
=(\lambda_g\otimes I_o) P_A.
\]
Since $\theta_AS_A^{-1/2}$ is a unitary operator in $B(H, P_AH_G)$ and since 
the unitary group $\{\pi_g \mid g\in G\}$ (resp., $(\lambda_g\otimes I_o)\left.P_A\right |_{P_A\negmedspace H_G}$) generate 
the von Neumann algebra $\pi(G)''$ (resp., $P_A(\mathscr L (G)\otimes I_o)\left.P_A\right |_{P_A\negmedspace H_G}$), the map
\[
\pi(G)" \ni X \mapsto \theta_AS_A^{-1/2}XS_A^{-1/2}\theta_A^* \in 
P_A(\mathscr L (G)\otimes I_o)\left.P_A\right |_{P_A\negmedspace H_G}
\]
 is a (spatial) isomorphism of von Neumann algebras.
Let $Q\otimes I_o$ be the central support of $P_A$, so $Q \in \mathscr R (G)\cap \mathscr L (G)$. 
It is well known \cite[Proposition 5.5.5.]{KR} that the map 
\[
(\mathscr L (G)\otimes I_o)\left.(Q\otimes I_o)\right |_{(Q\otimes I_o)H_G} 
\ni X \mapsto X\left.P_A\right |_{P_A\negmedspace H_G}
\in P_A(\mathscr L (G)\otimes I_o)\left.P_A\right |_{P_A\negmedspace H_G}
\]
is also an isomorphism.
Thus $\theta_AS_A^{-1/2}TS_A^{-1/2}\theta_A^* = (Z\otimes I_o)P_A$, for some operator
$Z\in \mathscr L (G)$ for which $\left. (Z\otimes I_o)(Q\otimes I_o)\right |_{(Q\otimes I_o)H_G}$ 
is invertible and then we have $AT= L_e^*(Z\otimes I_o)P_A\theta_A$. By passing if necessary to 
$ZQ+Q^\bot \in \mathscr L (G)$, we can assume without loss of generality
 that $Z$ is invertible. If $T$ is unitary, we can similarly assume that $Z$ too is unitary.
 
Recall that the involution $J$ defined by $J(x\chi_e):=x^*\chi_e$
 for all $x\in \mathscr L (G)$ and then extended to $\ell^2(\mathbb J)$, establishes the conjugate
linear isomorphism of $\mathscr L (G)$ and $\mathscr R (G)$, 
 $\mathscr L (G) \ni x \mapsto JxJ\in \mathscr R (G).$
For all $h\in H_o$,
\[
(Z^*\otimes I_o)L_eh= Z^*\chi_e\otimes h= JZJ\chi_e\otimes h = (JZJ\otimes I_o)L_eh, 
\]
hence 
\[
AT =L_e^* (Z\otimes I_o)\theta_A
= L_e^* ((JZJ)^*\otimes I_o)\theta_A
= L_e^* (JZ^*J\otimes I_o)\theta_A.
\]
Let $Y:=  JZ^*J$ and $M:= (Y\otimes I_o)P_A$.  Then $Y\in \mathscr R (G)$, hence $M\in \mathscr R (G)\otimes B(H_o)$ and 
$M=MP_A$. Furthermore, $Y$  is invertible (unitary if $T$ and hence $Z$ are unitary), hence
$M^*M=P_A(YY^*\otimes I_o)P_A$ is invertible in $B(P_AH_G)$. 
Then by Theorem \ref{T:param gen}, $AT = L_e^* (Y\otimes I_o)\theta_A=L_e^* M\theta_A $
 is an operator frame generator for  $(G, \pi, H)$.
\end{proof}

A reformulation of statement (ii) is that if two \o*s with generators $A$ and $B$ are 
right-similar, then the (unique) similarity operator must belong to $\pi(G)'$. 
Using this fact, the characterization of right-similarity for general \o*s 
carries through easily for \o*s with a generator as follows.

\begin{proposition}\label{P: right group sim} Let $A$ and $B$ be frame generators with the values in the same space $H_o$ 
 for a unitary representation $(G, \pi, H)$. Then the following conditions are equivalent:
\item[(i)]  $B=AT$ for some invertible operator $T\in\pi(G)'$;
\item[(ii)]  $B\pi_{g^{-1}}=A\pi_{g^{-1}}T$ for all $g\in G$ for some invertible operator $T\in B(H)$;
\item[(iii)] $\theta_B=\theta_AT$ for some invertible $T\in B(H)$;
\item[(iv)] $P_A\theta_BS_A^{-1}\theta_A^*$ is invertible in $B(P_AH_G)$;
\item[iv')] $B=L_e^*M\theta_A$ for some $M\in \mathscr R (G)\otimes B(H_o)$ with  
$M=MP_A$ and such that $P_AMP_A$ is invertible in $B(P_AH_G)$; and
\item[(v)] $P_B=P_A.$ 

\end{proposition}

\bC{C: central P}
Let $A\in B(H,H_o)$ be an operator frame generator for $(G, \pi, H)$. Then all the operator frame generators for $(G, \pi, H)$ are left similar to $A$ if and only if $P_A\in (\mathscr L(G) \cap (\mathscr R(G))\otimes I_o$.
\eC
\bp
$P_A$ belongs  to the center $ (\mathscr L(G) \cap \mathscr R(G))\otimes I_o$ of $\mathscr R(G)\otimes B(H_o)$ if and only if there are no projections $\mathscr R(G)\otimes B(H_o)$ that are different but Murray-von Neumann equivalent to it.  By Lemma \ref {L:basic group} and Proposition \ref {P: right group sim} this is equivalent to the condition that any operator frame generator for $(G, \pi, H)$ is left similar to $A$
\ep

\bR{R:special case2}
\item[(i)]  Corollary \ref {C: central P} provides a higher multiplicity analog of Proposition 3.13 in \cite {HL00}
\item[(ii)] If the group $G$ is abelian, then so is  $\mathscr L(G)=\mathscr R(G)$ and hence  $\mathscr R(G)\otimes I_o$ is the center of $\mathscr R(G)\otimes B(H_o)$. Thus in particular if $\dim H_o = 1$, $\mathscr R(G)\otimes B(H_o)$ is abelian and all operator frame generators for $(G, \pi, H)$ are are left similar. This generalizes Cor. 3.14 (and Theorem 6.3) of \cite{HL00} which states that, for vector group-frames, if $G$ is abelian then all group frames for a unitary representation of $G$ on the same Hilbert
space are unitarily equivalent. (Se also Remark \ref {R: special case} (i) in the present article.)
\eR
To simplify notations, we formulate the next result directly for Parseval operator frame generators.
\begin{proposition}\label{P:group simil}
Let $A,\,B\in B(H,H_o)$ be  Parseval  frame generators for $(G, \pi, H)$. 
\item[(i)] $B=AU$ for some unitary $U\in \pi(G)'$ if and only if $B=L_e^*W\theta_A$ for some unitary 
$W\in \mathscr R (G)\otimes B(H_o)$ with $WP_A=P_AW$, again if and only if $P_B=P_A$.
\item[(ii)] Let $U\in B(H_o)$  be unitary. Then $B=UA$ if and only if $B=L_e^*(I\otimes U)\theta_A$. If 
 $B=UA$, then $P_B=(I\otimes U)P_A(I\otimes U^*)$.
\item[(iii)] $B=AU$ for some unitary $U\in \pi(G)"$ if and only if $B=L_e^*(W\otimes I_o)\theta_A$
 for some unitary $W\in \mathscr R (G)$. If $B=L_e^*(W\otimes I_o)\theta_A$, then $P_B=(W\otimes I_o)P_A(W^*\otimes I_o)$.

\end{proposition}
\begin{proof}
\item[(i)] By Proposition \ref{P:simil}, $P_B=P_A$ if and only if the \o*s $\{B\pi_{g^{-1}}\}_{g\in G}$ and 
$\{A\pi_{g^{-1}}\}_{g\in G}$ are right unitarily equivalent. By Lemma \ref{L: 1}, this unitary equivalence
holds if and only if $B=AU$ for some unitary $U\in \pi(G)'$. Also, by Proposition \ref {P:simil},
 $P_B=P_A$ if and only if $\theta_B\theta_A^*$ is a unitary in $B(P_AH_G)$ and
 hence it is the compression to $P_AH_G$ of a unitary $W\in \mathscr R (G)\otimes B(H_o)$ that commutes with $P_A$.
\item[(ii)] It is immediate from Lemma \ref{L:left simil}
\item[(iii)] Assume that $B=AU$ for a unitary $U\in \pi(G)"$. In the proof of Lemma \ref{L: 1} (iii) we 
can choose $Z$ to be unitary and hence $W:=JZ^*J \in \mathscr R (G)$ is also unitary. Then 
$\theta_B= (W\otimes I_o)\,\theta_A$, hence $P_B=(W\otimes I_o)P_A(W^*\otimes I_o)$. On the other hand, if 
$B=L_e^*(W\otimes I_o)\theta_A$ for some unitary $W\in \mathscr R (G)$, then by the same argument as in the proof
of Lemma \ref{L: 1} (iii) we see that 
\[
B=L_e^*(JW^*J\otimes I_o)\,\theta_A= L_e^*\theta_A(\theta_A^*(JW^*J\otimes I_o)\theta_A),
\]
where $JW^*J\in \mathscr L (G)$ and hence  $\theta_A^*(JW^*J\otimes I_o)\theta_A$ is a unitary in $\pi(G)"$.
\end{proof}

\begin{remark}\label{R: other param}
For vector frames ($\dim H_o=1$), Han and Larson \cite[Theorem 6.17]{HL00} have shown that given
a Parseval frame generator $\eta$ for $(G, \pi, H)$,
the collection of all the (vector) Parseval frame generators for $(G, \pi, H)$
is parametrized by the unitary group of $\pi(G)"$, namely, it coincides with  
$\{U\eta \mid U\in \pi(G)", U\, \text{ is unitary}\}$. This result is also a consequence
of Theorem \ref{T:param gen}, since the partial isometry $V$ intertwining $P_A$ and $P_B$ can
 be extended to a unitary $W$ because the von Neumann algebra $\mathscr R (G)$ is finite. However, Proposition \ref{P:group simil}
shows why this result does not hold when $\dim H_o > 1$.
 
\end{remark}

\section{Homotopy of Operator Frame Generators}\label{S:homot}

The objective of this section is to prove the following theorem:

\begin{theorem}\label{T:homot} Let $(G, \pi, H)$ be a unitary representation of a discrete group $G$ and assume that  the collection $\mathscr F_G$ of all the operator frame generators
with range in $H_o$  for $(G, \pi, H)$ is non-empty. 
\item [(i)] If $\dim H_o < \infty$, then $\mathscr F_G$ is  path connected in the norm topology. 
\item [(ii)] If $\dim H_o = \infty$, then $\mathscr F_G$ is path connected in the norm topology if and only if the von Neumann algebra $\mathscr R (G)$ generated by the right regular representation of $G$ is diffuse, (i.e. has no nonzero minimal projections.)
\end{theorem}

As we will point out in the proof of the theorem, it is easy to reduce the problem to showing than any two Parseval  operator frame generators are homotopic. The latter property is obviously true in the case when $\dim H_o=1$ because then (by  \cite[Theorem 6.17]{HL00}, see also Remark \ref{R: other param} and  Proposition \ref{P:group simil},)
the collection of Parseval (vector) frame generators for $(G, \pi, H)$ is parametrized by the unitary group
 of the von Neumann algebra  $\pi(G)"$, which is well known to be path connected in the norm topology. 
In the general case, however, by Theorem \ref {T:param gen} (iii)  all Parseval operator frame 
generators for $(G, \pi, H)$ are parametrized by the partial isometries of the algebra  $\mathscr R (G) \otimes B(H_o)$ 
that have the same initial projection, the frame projection of a fixed Parseval operator
 frame generator. When $\dim H_o=\infty$, this class of partial isometries is not path connected in the norm topology. 
It is, however, path connected in the strong operator topology when $\mathscr L (G)$ has no nonzero minimal projections, and this is sufficient for the path connectedness in the norm topology of the operat frame generators. In order to do that we need to introduce some notations and preliminary results.

Let $V, W $ be partial isometries in $\mathscr R (G) \otimes B(H_o)$ with the same initial projection, i.e.,  
$V^*V=W^*W$, and hence with range projections, $VV^*, \, WW^*$ Murray-von Neumann equivalent ($VV^*\sim WW^*$). We say that 
\[
V\approx W
\]
if there is a norm continuous path of partial isometries  
\[\{V(t)\in \mathscr R (G) \otimes B(H_o) \mid t\in [0,1]\}
\]
 such that $V^*(t)V(t) =V^*V$  for all $t\in [0,1]$, $V(0)=V$, and $ V(1) =W$. Clearly,  $\approx$ is an equivalence 
relation for the class of partial isometries that have the same initial projection. Adaptating the proof of \cite[Theorem 3.1]{MS02} with a slight change, yields the following equivalent characterization. For the readers' convenience we sketch the proof.

\bL{L: unitary} Let $V, W $ be partial isometries in $\mathscr R (G) \otimes B(H_o)$ with the 
same initial projection. Then  $V\approx W$ if and only if $VV^*$ and $WW^*$ are unitarily equivalent in $\mathscr R (G) \otimes B(H_o)$.
\eL
\bp If $VV^*$ and $WW^*$ are unitarily equivalent, then $(WW^*)^\bot \sim (VV^*)^\bot$, i.e., there is a partial isometry $Z\in \mathscr R (G) \otimes B(H_o)$ with $ZZ^*=(WW^*)^\bot$ and \linebreak $Z^*Z=(VV^*)^\bot$. Then a simple computation shows that the operator \linebreak $U:= WV^*+Z\in  \mathscr R (G) \otimes B(H_o)$ is unitary  and that $W=UV$. Since $U$ is homotopic in the norm topology to the identity, choose a norm continuous path of unitaries $U(t) \in \mathscr R (G) \otimes B(H_o)$ with $U(0)=I$ and $U(1)=U$. Then $V(t):=U(t)V$ is the required norm continuous path of partial isometries  with the same initial projection that joins $V(0)=V$ and $V(1)=W$. This establishes that $V\approx W$.
Conversely, if $V\approx W$ and $V(t)$ is a norm continuous path of partial isometries with the same initial projection that joins $V(0)=V$ with $V(1)=W$, then $P(t):= V(t)V(t)^*$ is a norm continuous path of projections joining $P(0)= VV^*$ with $P(1)=WW^*$. It is well-known that homotopy of projections implies unitary equivalence. 

\ep

If $\dim H_o = \infty$, there are partial isometries $V, W \in  \mathscr R (G) \otimes B(H_o)$  with the same initial projection but  with range projections that are not unitarily equivalent, e.g., $V= I\otimes I_o$ and $W=I\otimes Z$ with $Z\in B(H_o)$ a non unitary isometry.   By Lemma \ref {L: unitary}, $V$ and $ W$ cannot be joined by a norm continuous path of partial isometries all with the same initial projection. 
The norm continuity of such a path of partial isometries, however, is only sufficient but is not always necessary for the existence of  a norm continuous path of Parseval frame generators joining $L_e^*V$ with $L_e^*W$.  The existence of the latter is,  in view of Theorem \ref {T:param gen} (iii),  equivalent to the existence a path of partial isometries $V(t)$ joining $V$ and $W$ for which  $L_e^*V(t)$ is norm continuous. It is convenient to denote the existence of such a path by using the following notation:

Let $V, W $ be partial isometries in $\mathscr R (G) \otimes B(H_o)$ with the same initial projection. We say that
\be{e: equiv}
V\underset{e}{\sim} W
\ee 
if there is a  path of partial isometries  $\{V(t)\in \mathscr R (G) \otimes B(H_o) \mid t\in [0,1]\}$ such
 that $L_e^*V(t)$ is norm continuous, $V^*(t)V(t) =V^*V$  for all $t\in [0,1]$, $V(0)=V$, and $ V(1) =W$.

Clearly, $\underset{e}{\sim}$  is also an equivalence relation for partial isometries that have the same  initial projection and $V\approx W$ implies  $V\underset{e}{\sim} W$. 

A key ingredient in the proof is that a finite trace in a von Neumann algebra is strongly continuous 
(actually, $\sigma$-weakly, but we do not need this here). 

Denote by $\tau (X) = (X \chi _e, \chi_e)$ for $X\in \mathscr R (G)$, the normalized trace on $\mathscr R (G)$. Denote by 
$E:= \mathscr R (G) \otimes B(H_o) \mapsto B(H_o)$ the corresponding \emph{slice map}, namely, the 
bounded linear extension of the map
$E(X\otimes Y) = \tau(X)Y$ for all $X\in \mathscr R (G)$ and all $Y \in B(H_o)$. It is easy to see that the map $E$ is positive, that is, $E(Z) \ge 0$ when $Z\ge 0$, or, equivalently, $E(Z_1) \le E(Z_2)$ when $Z_1 \le Z_2$.  Also, the map $E$  is normal, that is,  $E(Z_{\gamma}) \uparrow E(Z)$ when $Z_{\gamma} \uparrow Z$, or equivalently, $E$ is $\sigma$-weakly continuous.

The bridge between trace and norm is given by the following lemma. 

\begin{lemma}\label{L:slice}
$E(Z) = L_e^*ZL_e$ for all $Z \in \mathscr R (G)\otimes B(H_o)$
\end{lemma}
\begin{proof}
It is enough to verify that the two maps agree on elementary tensors. Indeed, for all $h\in H_o$ and all 
$X\in \mathscr R (G),\, Y\in B(H_o)$ we have
\begin{align*}
L_e^*(X\otimes I_o)L_eh &=L_e^*X\chi_e\otimes h = L_e^*(\sum_{g\in G}(X\chi_e,\chi_g )\chi_g)\otimes h\\
&=L_e^*\sum_{g\in G}(\chi_g\otimes (X\chi_e,\chi_g )h) = (X\chi_e,\chi_e )h = \tau(X)h.
\end{align*}
Thus $L_e^*(X\otimes I_o)L_e = \tau(X)I_o$ and hence 
\[
L_e^*(X\otimes Y)L_e = L_e^*(I\otimes Y)(X\otimes I_oY)L_e = YL_e^*(X\otimes I_oY)L_e 
= \tau(X)Y= E(X\otimes Y).
\]
\end{proof}
Notice that the trace $\tau$ on $\mathscr R (G)$ is always finite, while the trace  $\tau \otimes tr$ on $\mathscr R (G)\otimes B(H_o)$ is  finite only if $\dim H_o < \infty$. Thus, given two partial isometries with the same initial projection, we want
to construct a strongly continuous path of partial isometries that connect them, where the strong convergence occurs in the 
first component of the tensor product. This will be achieved via the following key lemma. 

\begin{lemma}\label{L:homot}  Assume that $\dim H_o = \infty$ and that  $\mathscr R (G)$ has no minimal projections. Let  $R$ be a  projection in $\mathscr R (G)$ with $R\sim R^\bot$, let $\{Q_n\}_{1\le n \le N} $ with $N\le \infty$ be a collection of  infinite projections in $B(H_o)$, and let $I= \sum_{n=1}^N F_n $  be a decomposition of the identity  into  mutually orthogonal central projections $F_n \in \mathscr R (G)\cap \mathscr L (G)$.  Then there is a path of partial isometries  $\{ W(t) \mid t\in [0,1]\}$ in
 $\mathscr R (G) \otimes B(H_o)$ such that 
\item [(i)]  $L_e^*W(t)$ is  norm continuous
\item [(ii)] $W(t)^*W(t)=\sum_{n=1}^NF_n\otimes Q_n$ for all $t \in [0, 1]$,
\item [(iii)] $W(t)W(t)^*\le \sum_{n=1}^NF_n\otimes Q_n$ for all $t \in [0, 1]$
\item [(iv)] $W(0) =\sum_{n=1}^NF_n\otimes Q_n$, and
\item [(v)] $W(1)W(1)^*=\sum_{n=1}^N RF_n\otimes Q_n$. 
\end{lemma}
\begin{proof} The reduced von Neumann algebra $R\mathscr R (G)R:= R\mathscr R (G)\left.R\right |_{R\ell^2(G)}$  has no minimal projections, thus it contains a strongly continuous increasing net of projections $\{R(t)\mid t\in [0,1]\}$ with $R(0)=0$, $R(1) = R$. For instance, $R(t)$ can be obtained from the spectral resolution of a positive generator of a maximal abelian subalgebra of  $R\mathscr R (G)R$.
Since $R^{\bot } \sim R$, there is a unitary $U\in \mathscr R (G)$ such that $R^{\bot }=URU^*$.
 Since $Q_n$ is an infinite projection in $B(H_o)$, there exist partial isometries $S_{1,n}, S_{2,n} \in B(H_o)$
such that 
\[
S_{1,n}^*S_{1,n}=S_{2,n}^*S_{2,n}=Q_n \quad \text{and} \quad S_{1,n}S_{1,n}^*+S_{2,n}S_{2,n}^*=Q_n.
\]
Notice that $S_{1,n}$, $S_{2,n}$ are the generators of the Cuntz algebra $\mathscr O_2$ represented on $Q_nH_o$. Define for $ t\in [0,1]$
\[
W(t):=\sum_{n=1}^N\Big( (R(t) +UR(t)U^*)^{\bot}F_n\otimes Q_n+ R(t)F_n \otimes S_{1,n}+R(t)U^*F_n\otimes S_{2,n}\Big).
\]
By definition, $W(0) = \sum_{n=1}^NF_n\otimes Q_n$. Since  $R(t)\,\bot\, (R(t) +UR(t)U^*)^{\bot}$ for all $t\in [0,1]$, 
and $S_{2,n}^*S_{1,n}=S_{1,n}^*S_{2,n}=0$ for all $n$, it follows that\\
\begin{align*}
&W(t)^*W(t) \\
&= \sum_{n=1}^N\Big((R(t) +UR(t)U^*)^{\bot}F_n\otimes Q_n +(R(t) +UR(t)U^*)^{\bot}R(t)F_n\otimes Q_nS_{1,n}\\
&+(R(t) +UR(t)U^*)^{\bot}R(t)U^*F_n\otimes Q_nS_{2,n}
+R(t)(R(t) +UR(t)U^*)^{\bot}F_n\otimes S_{1,n}^*Q_n\\
&+R(t)F_n\otimes S_{1,n}^*S_{1,n}
+R(t)U^*F_n\otimes S_{1,n}^*S_{2,n}\\
&+UR(t)(R(t) +UR(t)U^*)^{\bot}F_n\otimes S_{2,n}^*Q_n\\
&+UR(t)F_n\otimes S_{2,n}^*S_{1,n}
+UR(t)U^*F_n\otimes S_{2,n}^*S_{2,n}\Big)\\
&=\sum_{n=1}^N\Big((R(t) +UR(t)U^*)^{\bot}F_n\otimes Q_n
+R(t)F_n\otimes S_{1,n}^*S_{1,n}\\
&+UR(t)U^*F_n\otimes S_{2,n}^*S_{2,n}\Big)\\
&= \sum_{n=1}^N\Big((R(t) +UR(t)U^*)^{\bot}F_n\otimes Q_n
+R(t)F_n\otimes Q_n
+UR(t)U^*F_n\otimes Q_n\Big)\\
&=\sum_{n=1}^NF_n\otimes Q_n.
\end{align*}
Using the fact that 
$UR(t)U^*\,\bot\, (R(t) +UR(t)U^*)^{\bot}$ and $UR(t)U^*\,\bot\,R(t)$ for all $t$, 
a similar computation yields
\begin{align*}
&W(t)W(t)^* \\
&=\sum_{n=1}^N\Big((R(t) +UR(t)U^*)^{\bot}F_n\otimes Q_n
+R(t)F_n\otimes S_{1,n}S_{1,n}^*
+R(t)F_n\otimes S_{2,n}S_{2,n}^*\Big)\\
&=\sum_{n=1}^N\Big(((R(t) +UR(t)U^*)^{\bot} +R(t))F_n\otimes Q_n\Big) \\
&\le \sum_{n=1}^N F_n\otimes Q_n.
\end{align*}
In particular, 
\[
W(1)W(1)^* = \sum_{n=1}^N((R +R^\bot)^{\bot} +R)F_n\otimes Q_n = \sum_{n=1}^NRF_n\otimes Q_n. 
\]
Thus $\{W(t) \mid t\in [0,1]\}$ is a path of partial isometries of $\mathscr R (G) \otimes B(H_o)$ that satisfies
conditions (ii), (iii), (iv), and (v). We now show that the condition (i) is also satisfied. 
Let $0\le t<t'\le 1$ and $\Delta R:= R(t') - R(t)$. Then 
\[
W(t') -W(t) =\sum_{n=1}^N\Big(  -(\Delta R+U\Delta RU^*)F_n\otimes Q_n+\Delta R F_n\otimes S_{1,n}+\Delta RU^*F_n\otimes S_{2,n}\Big).
\]
By using the facts that  $\Delta R \, \bot\, U\Delta RU^*$ and $Q_nS_{i,n}=S_{i,n}Q_n=S_{i,n}$ for $i=1,2$ and all $n$, we obtain
\begin{align*}
&(W(t') -W(t))(W(t') -W(t))^*\\
&=\sum_{n=1}^N\Big ((\Delta R+U\Delta RU^*)^2F_n\otimes Q_n
 -(\Delta R+U\Delta RU^*)\Delta R F_n\otimes S_{1,n}^*\\
& -(\Delta R+U\Delta RU^*)U\Delta R F_n\otimes S_{2,n}^*
-\Delta R(\Delta R+U\Delta RU^*)  F_n\otimes S_{1,n}\\  
&+(\Delta R)^2  F_n\otimes S_{1,n} S_{1,n}^*
+\Delta R U\Delta RF_n\otimes S_{1,n}S_{2,n}^*\\
&-\Delta RU^*(\Delta R+U\Delta RU^*)F_n\otimes S_{2,n}
+\Delta RU^*\Delta RF_n\otimes S_{2,n}S_{1,n}^*\\
&+(\Delta R)^2F_n\otimes S_{2,n}S_{2,n}^*\Big)\\
&=\sum_{n=1}^N\Big((\Delta R+U\Delta RU^*)F_n\otimes Q_n
-\Delta R F_n\otimes S_{1,n}^*
-U\Delta R F_n\otimes S_{2,n}^*\\
&-\Delta R F_n\otimes S_{1,n}
+\Delta R F_n\otimes S_{1,n}S_{1,n}^*
-\Delta R U^*F_n\otimes S_{2,n}
+ \Delta R F_n\otimes S_{2,n}S_{2,n}^*\Big)\\
&= \sum_{n=1}^N\Big((2\Delta R+U\Delta RU^*)F_n\otimes Q_n
-\Delta R F_n\otimes (S_{1,n}^*+S_{1,n})\\
&- U\Delta R F_n\otimes S_{2,n}^*
-\Delta RU^* F_n\otimes S_{2,n} \Big)
\end{align*}
Thus,
\begin{align*}
E\bigr(&(W(t')-W(t))(W(t')-W(t))^*\bigl)=\\
&= \sum_{n=1}^N\Big(\tau (\Delta RF_n)(3Q_n-(S_{1,n}+S_{1,n}^*))
-\tau(\Delta RU^*F_n)S_{2,n}-\tau(U\Delta RF_n)S_{2,n}^*\Big)\\
&\le \sum_{n=1}^N 7\tau (\Delta RF_n)I_o = 7\tau (\Delta R)I_o.
\end{align*}
Hence
\[
\Vert L_e^* W(t')-L_e^*W(t)\Vert ^2=\Vert E\bigr((W(t')-W(t))(W(t')-W(t))^*\bigl) \Vert \le 7\tau (\Delta R).
\]
Since $\mathscr R (G)$ is finite, $\tau (R(t))$ is continuous and hence  $L_e^* W(t)$
 is norm continuous, which concludes the proof. 
\end{proof}
                                                                                         
Now we can proceed to prove Theorem \ref{T:homot}
\begin{proof} 
It is well known that in any von Neumann algebra (or, more in general, unital $C^*$-algebra), positive invertible operator are homotopic to the identity. 
But then, the operator frame generator  $A$ for $(G, \pi, H)$ is homotopic to the  Parseval operator frame generator $AS_A^{-1/2}$ by Lemma \ref {L:basic group} (ii). Thus, to prove the path connectedness in the norm topology of $\mathscr F_G$ it is enough to prove that the  collection of Parseval operator frame generators for $(G, \pi, H)$ is path-connected in the norm topology.  By (\ref{eq:Parseval}), this collection is parametrized by 
\[
\{L_e^*V\theta_A \mid  V\in \mathscr R (G)\otimes B(H_o), V^*V= P_A \},
\]
and since $\theta_A$ is an isometry, we only need  to prove that $V\underset{e}{\sim} W$ for any two partial isometries $V$ and $W$ in $\mathscr R (G) \otimes B(H_o)$ with the same initial projection $P_A$, i.e., $V^*V=W^*W=P_A$.  

(i) The algebra $\mathscr R (G)\otimes B(H_o)$ is finite because both $\mathscr R (G)$ and  $B(H_o)$ are finite, hence the equivalence of the projections $VV^*$ and $WW^*$ implies their unitary equivalence. But then $V\approx W$ by Lemma \ref {L: unitary} , and hence $V\underset{e}{\sim} W$. This proves that  $\mathscr F_G$ is norm connected.

(ii) We prove first that the condition is necessary.  Assume that $\mathscr R (G)$ has a nonzero minimal projection $Q$. Then $Q$ belongs to a finite type $I$ subfactor of $\mathscr R (G)$. Indeed if $c(Q)$ is the central cover of $Q$, which is the smallest projection in the center $\mathscr L (G) \cap \mathscr R (G)$ of $\mathscr R (G)$ that majorizes $Q$, then $c(Q)$ is minimal in $\mathscr L (G) \cap \mathscr R (G)$. But then the reduced von Neumann algebra $\mathscr R (G)c(Q) :=  c(Q)\mathscr R (G)\left.c(Q)\right |_{c(Q)\ell^2(G)}$ is a factor,  it is finite because so is   $\mathscr R (G)$, and it is of type $I$ because it contains the minimal projection $Q$. Let $\{E_{i,j}\}_{1\le i,j\le n}$ be a set of matrix units for $\mathscr R (G)c(Q)$, i.e., $E_{i,j}^*=E_{j,i}$,  $E_{i,k} E_{h,j} = \delta_{h,k} E_{i,j}$ for all $1\le i,j,h, k\le n$, $\sum_{i=1}^n  E_{i,i} = c(Q)$, and $\mathscr R (G)c(Q)= \{\sum_{i,j=1}^n c_{i,j}E_{i,j} \mid c_{i,j}\in \mathbb C\}$. Then every element $Z$ in the factor $\mathscr R (G)c(Q)\otimes B(H_o)$ has the unique matricial form 
\be {e: matrix form}
 Z= \sum_{i,j=1}^n E_{i,j}\otimes Z_{i,j} \quad  \text{for } Z_{i,j}\in B(H_o).
\ee
Therefore, it is easy to see from Lemma \ref{L:slice} that 
\be {e: type I slice}
L_e^*ZL_e= \frac{1}{n}\sum _{i=1}^nZ_{i,i}.
\ee

Since $\mathscr R (G)c(Q)\otimes B(H_o)$ is an infinite type I factor, there is a proper isometry $V\in \mathscr R (G)c(Q)\otimes B(H_o)$, i.e., $V^*V=c(Q)\otimes I_o$ but $VV^*\ne c(Q)\otimes I_o$. To prove that $\mathscr F_G$ is not  path connected in the norm topology, it will be enough to show that the two Parseval \o* generators $L_e^*c(Q)\otimes I_o$ and $L_e^*V$ cannot be  connected by any norm continuous path of arbitrary \o* generators.

Assume otherwise, then by Theorem \ref {T:param gen} (ii), there is a path of operators $M(t)\in \mathscr R (G)\otimes B(H_o)$ with $M(t)c(Q)\otimes I_o=M(t)$ , $M(t)^*M(t)c(Q\otimes I_o$ is invertible in $\mathscr R (G)c(Q)\otimes B(H_o)$,  $M(0)=c(Q)$, $M(1)= V$, and $L_e^*M(t)$ is norm continuous. By (\ref {e: matrix form}), $M(t) = \sum_{i,j=1}^n E_{i,j}\otimes M_{i,j}(t)$ for a (unique) collection $M_{i,j}(t)\in B(H_o)$. Then for all $s, t\in [0,1]$, by (\ref {e: type I slice}),
\begin{align*}
(L_e^*M(s) -& L_e^*M(t))(L_e^*M(s) - L_e^*M(t))^*\\
&= L_e^*\Big( \sum_{i,j=1}^n E_{i,j}\otimes\sum_{k=1}^n (M(s)_{i,k} - M(t)_{i,k} )((M(s)_{k,j} - M(t)_{k,j} )^*\Big)L_e\\
&= \frac{1}{n}\sum _{i,k=1}^n (M(s)_{i,k} - M(t)_{i,k} )((M(s)_{k,i} - M(t)_{k,i} )^*.
\end{align*}
Thus the norm continuity of $L_e^*M(t)$ is equivalent to the norm continuity of  $M_{i,j}(t)\in B(H_o)$ for all $1\le i,j \le n$, and the latter is equivalent to the norm continuity of $M(t)$. But then, $M(t)^*M(t)$ is norm continuous and by the norm continuity of the inverse (e.g., see \cite [Problem 100] {pH82}), $(M(t)^*M(t))^{-1}$ is also norm continuous in  $\mathscr R (G)c(Q)\otimes B(H_o)$. As a consequence, $P(t): M(t)(M(t)^*M(t))^{-1}M(t)*$ is a norm continuous path, and it is immediate to see (cfr. Theorem \ref{T:param} ) that $P(t)$ are projections. But   this is impossible since $P(0)= c(Q)$ and $P(1)= VV^*$ are not unitarily equivalent and hence not homotopic.

We prove now that if $\mathscr R (G)$ has no nonzero minimal projections then $V \underset{e}{\sim} W$. By the standard type decomposition of von Neumann algebras (for these and other von Neumann algebra properties see \cite {KR}), there is a (unique) central projection $F^{(1)}\in \mathscr L (G)\cap \mathscr R (G)$ for which $(\mathscr L (G)\cap \mathscr R (G))F^{(1)}$ is diffuse, i.e., has no atoms and hence it is isomorphic to $L^\infty(\mathbb R)$) and  $(\mathscr L (G)\cap \mathscr R (G))(F^{(1)})^{\bot}$ is atomic and hence $\mathscr R (G))(F^{(1)})^{\bot}$ is a direct sum of type $II_1$ factors.  Since it is immediate to verify that $V \underset{e}{\sim} W$ if and only if both 
\[
VF^{(1)}\otimes I_o \underset{e}{\sim} WF^{(1)}\otimes I_o \quad \text {and} \quad V(F^{(1)})^{\bot}\otimes I_o \underset{e}{\sim} W(F^{(1)})^{\bot}\otimes I_o,
\]
we can consider separately the cases where the center of $\mathscr R (G)$ is diffuse and where it is atomic.

Consider first the case  where  $\mathscr R (G)$  has diffuse center. Then there is a strongly continuous increasing net of central projections $F(t)\in \mathscr L (G) \cap \mathscr R (G)$ such that $F(0)= 0$ and $F(1)=I$. Let
\[
V(t) := V(F(t)^\bot \otimes I_o )+ W(F(t) \otimes I_o).
\]
Since $F(t)\otimes I_o $ is in the center of $\mathscr R (G)\otimes B(H_o)$, we see that 
\[
V(t)^*V(t)= V^*V(F(t)^\bot\otimes I_o ) + W^*W(F(t) \otimes I_o ) = V^*V = P_A
\]
for all $t\in [0, 1]$ and $V(0) = V, V(1)= W$. Moreover, for all $s< t \in [0, 1]$, 
\[
V(t) - V(s)= V(F(s)-F(t))\otimes I_o + W(F(t)-F(s))\otimes I_o = (W-V)(F(t)-F(s))\otimes I_o,
\]
hence 
\begin{align*}
(V(t) - V(s))&(V(t) - V(s))^*\\
&= (F(t)-F(s))\otimes I_o\big( (W-V)(W-V)^* \big)(F(t)-F(s))\otimes I_o\\
& \le ||W-V||^2  (F(t)-F(s))\otimes I_o\\
& \le 4 (F(t)-F(s))\otimes I_o.
\end{align*}
But then, 
\begin{align*}
||L_e^*V(t) - L_e^*V(s)||^2 &= ||E((V(t) - V(s))(V(t) - V(s))^*)|| \\
&\le 4 ||E((F(t)-F(s))\otimes I_o)|| \\
&= 4 \tau ((F(t)-F(s)).
\end{align*}
By the strong (actually $\sigma$-weak)  continuity of $\tau$,  $L_e^*V(t)$ is norm continuous, and hence  $V\underset{e}{\sim} W$. Following the terminology introduced  in \cite{AILP} for wavelet generators for the unitary system, we say that the path constructed in the case where  $\mathscr R (G)$  has diffuse center is \emph{a direct path}.

 Consider now the key case where $\mathscr R (G)$ has no nonzero minimal projections and the center of $\mathscr R (G)$ is atomic. Then the identity $I\in\mathscr R (G)$ can be decomposed (uniquely) into a sum $I= \sum_{n=1}^N F_n$   of $N\le \infty$ mutually orthogonal projections  $F_n$ minimal in the center $\mathscr R (G)\cap\mathscr L (G)$. Notice that since $\mathscr R (G)$ has a finite faithful trace $\tau$, the decomposition is at most countable. The minimality of the projections $F_n$ implies that each reduced von Neumann algebra $\mathscr R (G)F_n: \mathscr R (G)F_n\mathscr R (G)\left.F_n\right.|_{F_n\ell^2(G)}$ is a factor, and being finite and with no minimal projections, it is of type $II_1$. 
Let 
\begin{align*}
&F^{(2)}:= \sum \{ F_n \mid V^*VF_n \otimes I_o \not\sim F_n\otimes I_o\}\\
&F^{(3)}:= \sum \{ F_n \mid (VV^*)^{\bot}F_n\otimes I_o \sim F_n\otimes I_o, F_n \le (F^{(2)})^{\bot}\}\\
&F^{(4)} := \sum \{ F_n \mid (VV^*)^{\bot}F_n\otimes I_o \not\sim F_n\otimes I_o, F_n \le (F^{(2)})^{\bot}\}.
\end{align*}

Thus $F^{(2)}+F^{(3)}+F^{(4)} = I$.
Reasoning as above, we can consider separately the cases where $F^{(2)}=I$  and $F^{(2)}=0$

Assume first that $F^{(2)}=I$, i.e., 
\[
W^*WF_n \otimes I_o = V^*VF_n \otimes I_o \not\sim F_n\otimes I_o \quad \text{ for all } n.
\]
Then 
\be{e: not equiv}
WW^*F_n\otimes I_o\not  \sim F_n\otimes I_o \quad \text{and}\quad VV^*F_n\otimes I_o \not\sim F_n\otimes I_o \quad \text{ for all } n.
\ee
Since   the factor  $\mathscr R (G)\otimes B(H_o)$ is infinite, (\ref {e: not equiv}) implies  that  $(VV^*)^{\bot}F_n \otimes I_o \sim F_n\otimes I_o$ for every $n$, and hence $(VV^*)^{\bot}\sim I\otimes I_o$ and similarly, $(WW^*)^{\bot}\sim I\otimes I_o$. But then $VV^*$ and $WW^*$ are unitarily equivalent, hence $V\approx W$  by Lemma \ref{L: unitary},  and thus $V\underset{e}{\sim} W$.

When $F^{(2)}=0$, i.e., $F^{(3)}+F^{(4)} =I$,  we have $W^*W= V^*V\sim I\otimes I_o$. Since $\mathscr R (G)$ is a direct sum of type $II_1$ factors and in every $II_1$ factor there are projections equivalent to their orthogonal complement, we can fix a projection $R\in \mathscr R (G)$ with $R \sim R^\bot$. As in each of the infinite factors we have $RF_n\otimes I_o \sim R^{\bot}F_n\otimes I_o \sim F_n\otimes I_o$, it follows that  $R\otimes I_o \sim R^\bot\otimes I_o \sim I\otimes I_o$. Fix a partial  isometry $V_o \in \mathscr R (G) \otimes B(H_o)$ with $V_o^*V_o=V^*V = W^*W$ and $V_o^*V_o=R\otimes I_o$.  We claim that $V\underset{e}{\sim} V_o$. The same argument  will prove that $W\underset{e}{\sim} V_o$ and hence that $V\underset{e}{\sim} W$, which will conclude the proof.

Assume next that $F^{(3)}=I$, then 
\[
VV^*\sim V^*V\sim I\otimes I_o\sim R\otimes I_o\quad \text{and}\quad (VV^*)^{\bot}\sim I\otimes I_o\sim R^{\bot}\otimes I_o= (R\otimes I_o)^{\bot}. 
\]
Thus  $VV^*$ and $V_o^*V_o$ are unitarily equivalent, hence $V \approx V_o$, and hence $V\underset{e}{\sim} V_o$.

Finally consider the case where $F^{(4)}=I$, namely where $V^*V\sim I\otimes I_o$ but $(VV^*)^{\bot}F_n\otimes I_o \not\sim F_n\otimes I_o$ for every $n$, which is the crux of the proof.
If for a certain $n$ the projection $(VV^*)^{\bot}F_n\otimes I_o$ is finite and hence $(\tau\otimes tr)((VV^*)^{\bot}F_n\otimes I_o)< \infty$, let 
$Q_n^\bot \in B(H_o)$  be a finite projection with  
\[
tr(Q_n^\bot) > \frac{(\tau\otimes tr)((VV^*)^{\bot}F_n\otimes I_o)}{\tau (F_n)}. 
\]
Then $Q_n\sim I_o$. Since $\mathscr R (G)F_n$ is a type $II_1$ factor, it contains a projection $R_n$   with trace
\[
\tau (R_n) = \frac{(\tau\otimes tr)((VV^*)^{\bot}F_n\otimes I_o)}{\tau(F_n)tr(Q_n^\bot )},
\]
Equivalently, 
\[                                                                                                                               
(\tau\otimes tr) (R_n\otimes Q_n^\bot ) = (\tau\otimes tr)((VV^*)^{\bot}F_n\otimes I_o)
\]
and hence, again because  $\mathscr R (G)F_n$ is a factor,  $R_n\otimes Q_n^\bot \sim (VV^*)^{\bot}F_n\otimes I_o$.
If for a certain $n$ the projection  $(VV^*)^{\bot}F_n\otimes I_o$ is infinite (but still $(VV^*)^{\bot}F_n\otimes I_o \not\sim F_n\otimes I_o$ by the definition of  $F^{(4)}$), there is a projection $Q_n^\bot \in B(H_o)$ with $Q_n\sim I_o$ and for which $F_n\otimes Q_n^\bot \sim (VV^*)^{\bot}F_n\otimes I_o$.  In this case, set $R_n:=F_n$, so for all $n$, 
\[
(VV^*)^{\bot}F_n\otimes I_o \sim R_n\otimes Q_n^{\bot}
\]
with $R_n= R_nF_n$. Moreover, 
\[
(R_n \otimes Q_n^\bot)^\bot F_n\otimes I_o= F_n\otimes Q_n + R_n^\bot F_n \otimes Q_n^\bot
\sim F_n \otimes I_o \sim VV^*F_n
\]
because  $Q_n\sim I_o$ and $Q_n^\bot \not \sim I_o$.
Thus $VV^*F_n\otimes I_o$ and $F_n\otimes Q_n + R_n^\bot F_n \otimes Q_n^\bot$ are unitarily equivalent for all $n$ and hence $VV^*$ and $\sum_{n=1}^NF_n\otimes Q_n+\sum_{n=1}^N R_n^\bot F_n\otimes Q_n^\bot$ are unitarily equivalent in $\mathscr R (G)\otimes B(H_o)$. Choose a unitary 
$U\in \mathscr R (G)\otimes B(H_o)$ such that $UVV^*U^* = \sum_{n=1}^NF_n\otimes Q_n+\sum_{n=1}^N R_n^\bot F_n\otimes Q_n^\bot$. Then $V\approx UV$.

Now apply Lemma \ref{L:homot} to the fixed projection $R\sim R^\bot$, and the sequences  of central projections $F_n$ and infinite projection $Q_n$ that we have constructed for $1\le n \le N$.  Thus we obtain  a path of partial isometries
 $\{ W(t) \mid t\in [0,1]\}$ in $\mathscr R (G) \otimes B(H_o)$ where $L_e^*W(t)$ is  norm continuous, 
 $W(t)^*W(t)=\sum_{n=1}^NF_n\otimes Q_n$ and  $W(t)W(t)^*\le \sum_{n=1}^NF_n\otimes Q_n$ for all $t \in [0, 1]$, $W(0) =\sum_{n=1}^NF_n\otimes Q_n$, and $W(1)W(1)^*=\sum_{n=1}^N RF_n\otimes Q_n$.
Then let 
\[
V(t) := \Big(W(t) + \sum_{n=1}^N R_n^\bot F_n\otimes Q_n^\bot\Big) UV \quad \text{for}\quad t\in [0,1].
\]
Since the initial projections and the range projections of all the partial isometries $W(t)$ are orthogonal 
to  $\sum_{n=1}^N R_n^\bot F_n\otimes Q_n^\bot$,  $W(t) + \sum_{n=1}^N R_n^\bot F_n\otimes Q_n^\bot$ are partial isometries for all $ t\in [0,1]$. Therefore 
$V(t)$ are also partial isometries in $\mathscr R (G) \otimes B(H_o)$, and all have initial projection $P_A = (UV)^*UV$. 
Since $L_e^*W(t)$ is norm continuous, so is $L_e^*V(t)$. 
Thus by definition, $UV = V(0)  \underset{e}{\sim} V(1)$. Moreover, the range projection of $V(1)$ is unitarily equivalent  to the range projection $R\otimes I_o$ of $V_o$. Indeed, 
\begin{align*}
V(1)V(1)^* &= (W(1) +  \sum_{n=1}^N R_n^\bot F_n\otimes Q_n^\bot UVV^*U^*(W(1)^* +  \sum_{n=1}^N R_n^\bot F_n\otimes Q_n^\bot)\\
&= W(1)W(1)^* +  \sum_{n=1}^N R_n^\bot F_n\otimes Q_n^\bot \\
&=  \sum_{n=1}^N RF_n\otimes Q_n+ \sum_{n=1}^N R_n^{\bot} F_n\otimes Q_n^\bot\\
& \sim R\otimes I_o
\end{align*}
and   
\begin{align*}
(V(1)V(1)^*)^\bot =   \sum_{n=1}^N R^\bot F_n\otimes Q_n+  \sum_{n=1}^N R_n^\bot F_n \otimes Q_n^\bot \sim R^\bot\otimes I_o =(R\otimes I_o)^\bot.
\end{align*}

But then, $V(1) \approx V_o$. Since we have already established that $V\approx UV \underset{e}{\sim} V(1)$, we conclude in
 this case too that $V \underset{e}{\sim} V_o$, thus completing the proof.
 \ep

\end{document}